\documentclass[11pt]{article}
\usepackage{amssymb,amsmath,accents} 
\usepackage{amsfonts,amsthm,mathrsfs}
\usepackage[pdftex]{graphicx}
\usepackage{caption}   
\usepackage{subcaption} 
\usepackage{float}     
\usepackage{enumitem}
\usepackage[dvipsnames]{xcolor}

\numberwithin{equation}{section}
\newtheorem{thm}{Theorem}[section]
\newtheorem{lemma}{Lemma}[section]
\newtheorem{definition}{Definition}[section]
\newtheorem{rem}{Remark}[section]

\newtheorem{cor}{Corollary}[section]

\newcommand{\Lemma}[1]{Lemma~\ref{#1}}
\newcommand{\Theorem}[1]{Theorem~\ref{#1}}
\newcommand{\Corollary}[1]{Corollary~\ref{#1}}
\newcommand{\Remark}[1]{Remark~\ref{#1}}
\newcommand{\Definition}[1]{Definition~\ref{#1}}

\newcommand{\Section}[1]{Section~\ref{#1}}
\newcommand{\Figure}[1]{Fig.~\ref{#1}}
\newcommand{\closure}[1]{\overline{#1}}
\newcommand{\dH}[1]{\mathrm{d}\mathcal{H}^{#1}}
\newcommand{\dL}[1]{\mathrm{d}\mathcal{L}^{#1}}
\newcommand{\mL}[1]{\mathcal{L}^{#1}}
\newcommand{\mH}[1]{\mathcal{H}^{#1}}
\newcommand{\bv}[1]{\mathbf{#1}}
\newcommand{\moll}[1]{{\rho}_{#1}}
\newcommand{\mollHalf}[1]{\rho^+_{#1}}
\newcommand{\mollD}[2]{\rho^{(#1)}_{#2}}
\newcommand{\mollDHalf}[2]{\rho^{(#1+)}_{#2}}
\newcommand{\mollDTHalf}[2]{\tau^{(#1+)}_{#2}}
\newcommand{\conv}[2]{\left(#1*#2\right)}
\newcommand{\convS}[2]{#1*#2}
\newcommand{\convDret}[4]{#1\ast_{#3,#4}#2}
\newcommand{\chgval}[0]{\Psi_\Omega}

\newcommand{\commutaor}[2]{r_\eta(#1,#2)}
\newcommand{\halfspace}[1]{\mathbb{R}^{#1}_+}
\newcommand{\divergence}[1]{\operatorname{div}#1}

\begin{document}
\title{On existence and uniqueness for transport equations with non-smooth velocity fields under inhomogeneous Dirichlet data}
\author{Tokuhiro Eto\thanks{Graduate School of Mathematical Sciences, The University of Tokyo, Komaba 3-8-1, Meguro, Tokyo 153-8914, Japan. E-mail:tokuhiro.eto@gmail.com} \and Yoshikazu Giga \thanks{Graduate School of Mathematical Sciences, The University of Tokyo, Komaba 3-8-1, Meguro, Tokyo 153-8914, Japan. E-mail:labgiga@ms.u-tokyo.ac.jp}}
\date{}

\maketitle
\thispagestyle{empty}
\begin{abstract}
A transport equation with a non-smooth velocity field is considered under inhomogeneous Dirichlet boundary conditions.
The spatial gradient of the velocity field is assumed in $L^{p'}$ in space and the divergence of the velocity field is assumed to be bounded.
By introducing a suitable notion of solutions, it is shown that there exists a unique renormalized weak solution for $L^p$ initial and boundary data for $1/p+1/p'=1$.
Our theory is considered as a natural extension of the theory due to DiPerna and Lions (1989), where there is no boundary.
Although a smooth domain is considered, it is allowed to be unbounded. A key step is a mollification of a solution.
In our theory, mollification in the direction normal to the boundary is tailored to approximate the boundary data.
\end{abstract}
{\small
}
\section{Introduction} \label{SD}
We consider a transport equation with non-smooth vector field $\bv{b}$ under inhomogeneous Dirichlet data in a smooth domain $\Omega$ in $\mathbb{R}^d$ ($d\ge2$) which is possibly unbounded. Namely, we consider
\begin{align}
	u_t + \bv{b}\cdot\nabla u &= 0 \quad\mbox{in}\quad \Omega\times(0,T), \label{ETr} \\
	u &= h \quad\mbox{on}\quad \partial\Omega\times(0,T), \label{EDir} \\
	u(\cdot,0) &= u_0 \quad\mbox{on}\quad \Omega, \label{EInit}
\end{align}
where $T>0$ is a fixed time horizon.
The velocity field $\bv{b}:\Omega\times(0,T)\to\mathbb{R}^d$ is given, but it may not be Lipschitz in space.
It is assumed to be in $L^1\left(0,T;W^{1,p'}(\Omega,\mathbb{R}^d)\right)$.
The function $h:\partial\Omega\times(0,T)\to\mathbb{R}$ is the boundary data, and the function $u_0:\Omega\to\mathbb{R}$ is the initial data.
We only assume that $u_0\in L^p(\Omega)$ and $h\in L^\infty\left(0,T;L^p(\partial\Omega)\right)$ for $1/p+1/p'=1$, $1\le p<\infty$.
We expect that the solution $u$ is only in $L^\infty\left(0,T;L^p(\Omega)\right)$, which is a weak solution.
It is impossible to take traces of $u$ to the boundary and also at $t=0$.
We begin with a suitable notion of a weak solution to \eqref{ETr}, \eqref{EDir}, and \eqref{EInit}. 

\begin{definition} \label{DDir}
{
	Let $h\in L^\infty(0,T;L^p(\partial\Omega))$, $u_0\in L^p(\Omega)$, and $\bv{b}\in L^1(0,T;W^{1,p'}(\Omega;\mathbb{R}^d))$ with $1\leq p < \infty$,
	where $p'$ denotes the conjugate number of $p$, say $1/p + 1/p' = 1$. Then, a}
function $u\in L^\infty\left(0,T;L^p(\Omega)\right)$ is said to be a weak solution of \eqref{ETr}, \eqref{EDir} and \eqref{EInit} provided that
\begin{multline*}
	-\int_0^T\int_\Omega u\partial_t\varphi\,\dL{d}dt - \int_\Omega u_0\varphi(\cdot,0)\,\dL{d} \\ + \int_0^T\int_{\partial\Omega}h(\bv{b}\cdot\nu_\Omega)\varphi\,{\dH{d-1}}dt - \int_0^T\int_\Omega u\divergence{(\varphi\bv{b})}\,\dL{d}dt = 0
\end{multline*}
holds for all $\varphi\in C_0^1\left(\closure{\Omega}\times[0,T)\right)$.
Here, $\mL{d}$ and $\mH{d-1}$ respectively denote the $d$-dimensional Lebesgue measure and the $(d-1)$-dimensional Hausdorff measure;
$\nu_\Omega$ denotes the outward unit normal vector field on $\partial\Omega$.
It is easily seen by the integration by parts that a classical solution to \eqref{ETr}, \eqref{EDir} and \eqref{EInit} is a weak solution in the sense of \Definition{DDir}.
\end{definition}

We note that the trace of $\bv{b}$ on $\partial\Omega$ belongs to $W^{1-\frac{1}{p'},p'}(\partial\Omega)$ which is continuously embedded into $L^r(\partial\Omega)$
with $r = (d-1)p'/(d-p')$ by the Sobolev embedding theorem.
Hence, the assumption that $h(\cdot,t)\in L^p(\partial\Omega)$ can be relaxed to $h(\cdot,t)\in L^s(\partial\Omega)$ with $s := \max\{1, p(d-1)/d\}$ to define the boundary integral in \Definition{DDir}.
However, for simplicity, we suppose that $h(\cdot,t)\in L^p(\partial\Omega)$ throughout this paper.

In this paper, we establish a reasonable notion of a weak solution to the transport equation with inhomogeneous Dirichlet boundary conditions in regular domains \eqref{ETr}, \eqref{EDir} and \eqref{EInit}.
By reasonable, we mean that the weak solution is unique regardless of the boundary data on the place where the velocity vector field goes outside.
Our aim is to extend DiPerna--Lions' theory to the case when an ambient space for the problem is specified and the Dirichlet boundary condition is inhomogeneous.
Let us state our main result:
\begin{thm}[{Uniqueness}] \label{PRDir}
Assume that $\Omega$ {is a $C^3$ domain (not necessarily bounded) in $\mathbb{R}^d$ and is uniformly-$C^2$.}
Let $1\leq p < \infty${, $h\in L^\infty(0,T;L^p(\partial\Omega))$, $u_0\in L^p(\Omega)$} and $\bv{b}\in L^1(0,T;W^{1,{p'}}(\Omega;\mathbb{R}^d))$ with $\divergence{\bv{b}}\in L^\infty(\Omega\times(0,T))$.
Then, the problem \eqref{ETr}, \eqref{EDir} and \eqref{EInit} has at most one weak solution (irrelevant to the value of $h$ on the place where $\bv{b}\cdot\nu_\Omega\geq 0$) in the class $L^\infty(0,T;L^p(\Omega))$.
\end{thm}

For the definition of that $\Omega$ is uniformly-$C^2$, we refer the reader to Bolkart and Giga \cite[\S 2]{BG}. Roughly speaking,
the boundary $\partial\Omega$ is locally the graph of a $C^2$ function whose second derivatives are uniformly bounded.
It is sufficient for ensuring this to assume that $\partial\Omega$ has $C^{2,1}$ regularity.
We note that if $\Omega$ is bounded and {$C^2$},
then it {is uniformly-$C^2$}.

{
	Meanwhile, the existence of a weak solution to \eqref{ETr}, \eqref{EDir} and \eqref{EInit} follows from a standard Gronwall-type argument;
	for the reader's convenience, we will prove this in \Section{sec:exst}:
	\begin{thm}[Existence]\label{thm:exst_2}
		Assume that $\Omega$ is a smooth domain in $\mathbb{R}^d$ (not necessarily bounded),
		and that $\divergence{\bv{b}}\in L^\infty(\Omega\times(0,T))$ and $\bv{b}\in L^\infty(\Omega\times(0,T))\cap L^1(0,T;W^{1,p'}(\Omega;\mathbb{R}^d))$.
		Then, there exists a weak solution $u\in L^\infty(0,T;L^p(\Omega))$ to \eqref{ETr}, \eqref{EDir} and \eqref{EInit}.
	\end{thm}
}

Our research is inspired by the theory proposed by DiPerna and Lions in 1989 \cite{DiPernaLions1989}.
This theory showed that a distributional solution to the transport equation in the Euclidean spaces is a \textit{renormalized solution},
and this leads to the existence and the uniqueness of the weak solution together with its stability.
We refer the reader for an introduction of DiPerna--Lions' theory to the book by Giga and the second author \cite[\S 2]{GigaGiga2024}.

The regularity assumptions on a weak solution and an associated velocity vector field are important to establish the uniqueness of a weak solution.
Indeed, it has been drawing the attention of researchers in what function spaces a weak solution and a velocity vector field should be sought to ensure the uniqueness of a weak solution.
We briefly summarize the previous results on the non-uniqueness of a weak solution to the transport equation.
Modena and Sz\'ekelyhidi \cite[Collorary 1.3]{Modena2018} showed that the uniqueness of a weak solution to the transport equation fails
in the solution space $C(0,T;L^p(\mathbb{T}^d))$ if $\bv{b}$ belongs to $C(0,T;W^{1,q}(\mathbb{T}^d;\mathbb{R}^d))$ where
$\mathbb{T}^d := \prod_{i=1}^{d}(\mathbb{R}/\omega_i\mathbb{Z})\,(\omega_i > 0,\, 1\leq i\leq d)$ denotes the $d$-dimensional torus, and $p$, $q$ and $d$ satisfy
\begin{equation*}
	\frac{1}{p}+\frac{1}{q} >1+\frac{1}{d-1},\quad p>1,\quad \mbox{and}\quad d\geq 3.
\end{equation*}
This result was improved by Modena and Sattig \cite{ModenaSattig2020}, and it was shown that the uniqueness fails if
\begin{equation}\label{eq:modena}
	\frac{1}{p}+\frac{1}{q}>1+\frac{1}{d},\quad p>1,\quad \mbox{and}\quad d\geq 2.
\end{equation}
Meanwhile, Bru\'e, Colombo, and De Lellis \cite{Brue2021} proved that there are possibly two positive solutions to the transport equation if $p$, $q$ and $d$ satisfy the {condition \eqref{eq:modena}} as in \cite{ModenaSattig2020}.
Cheskidov and Luo \cite[Theorem 1.3]{CheskidovLuo2021} investigated the critical threshold of $1/p+1/q$ to guarantee the uniqueness result by DiPerna--Lions,
and they showed that the uniqueness of a weak solution to the transport equation fails in the solution space $L^1(0,T;L^p(\mathbb{T}^d))$ as soon as
\begin{equation*}
	\frac{1}{p} + \frac{1}{q} > 1,\quad p>1,\quad \mbox{and}\quad d\geq 3
\end{equation*}
for some velocity vector field $\bv{b}\in L^1(0,T;W^{1,q}(\mathbb{T}^d;\mathbb{R}^d))\cap L^\infty(0,T;L^{p'}(\mathbb{T}^d;\mathbb{R}^d))$.
Let us summarize these results. The non-uniqueness of a weak solution to the transport equation has been shown in the solution space $L^\infty(0,T;L^p(\mathbb{T}^d))$ in the case when $1/p+1/q > 1+1/d$.
(Note that $C(0,T;L^p(\mathbb{T}^d))\subset L^\infty(0,T;L^p(\mathbb{T}^d))$, and hence the non-uniqueness result in \eqref{eq:modena} implies that in $L^\infty(0,T;L^p(\mathbb{T}^d))$.)
Meanwhile, extending the solution space to $L^1(0,T;L^p(\mathbb{T}^d))$ leads to the non-uniqueness result even if $1/p+1/q > 1$.
However, it is still an open problem whether we can deduce such a result if
\begin{equation*}
	\frac{1}{p} + \frac{1}{q} = 1\qquad\mbox{in}\quad L^1(0,T;L^p(\mathbb{T}^d)).
\end{equation*}
The same authors \cite[Theorem 1.2]{CheskidovLuo2024} recently showed that the scaling {condition on} $p$ and $q$ is irrelevant to prove the non-uniqueness of a weak solution to the transport equation.
Precisely speaking, they showed that there exists a velocity vector field $\bv{b}\in \bigcap_{p<\infty}L^1(0,T;W^{1,p}(\mathbb{T}^d;\mathbb{R}^d))$ such that
{a solution is not unique in} the function space
\begin{equation*}
	\mathcal{F}_{\bv{b}} :=\left\{u\in \bigcap_{\substack{p < \infty\\k\in\mathbb{N}}}L^p(0,T;C^k(\mathbb{T}^d))\biggm| u\bv{b}\in L^1(\mathbb{T}^d\times(0,T);\mathbb{R}^d)\right\}{.}
\end{equation*}
A method of convex integration is invoked to construct several weak solutions with the same data.

In the Lagrangian framework, it is well-known that the initial-value problem \eqref{ETr} and \eqref{EInit} amounts to finding a flow map $\mathbf{X}:\mathbb{T}^d\times[0,T)\to\mathbb{T}^d$ satisfying
\begin{equation}\label{eq:Lflow}
	\begin{cases}
		\partial_t{\mathbf{X}}(x,t) = \bv{b}(\mathbf{X}(x,t),t)&\quad\mbox{for}\quad (x,t)\in\mathbb{T}^d\times(0,T),\\
		\mathbf{X}(x,0) = x&\quad\mbox{for}\quad x\in\mathbb{T}^d.
	\end{cases}
\end{equation}
In this direction, Crippa and De Lellis \cite{Crippa2008} showed the well-posedness of the initial-value problem \eqref{eq:Lflow} for velocity fields $\bv{b}\in C(0,1;W^{1,r}(\mathbb{T}^d;\mathbb{R}^d))$ with $r > d$.
Meanwhile, Pitcho and Sorella \cite[Theorem 1.3]{Pitcho2023} showed that
for every $r\in [1,d)$ and $s<\infty$, there exists a divergence-free velocity field $\bv{b}\in C(0,1;W^{1,r}(\mathbb{T}^d;\mathbb{R}^d)\cap L^s(\mathbb{T}^d;\mathbb{R}^d))$
for which it is possible to find any finite number of integral curves starting from a.e. $x\in \mathbb{T}^d$ at $t = 0$.

For weaker regularity of the velocity field than Sobolev spaces, we refer to the work by Ambrosio \cite{Ambrosio2004}.
Therein, the velocity field was merely of locally bounded variation, 
and a measure-theoretic solution to the transport equation was shown to be a renormalized solution together with its existence and uniqueness.
However, it was not clear what kind of ambient spaces were used in this work, and hence the boundary condition and its effect were not treated.

For bounded domain problems, Tenan \cite{Tenan2021} studied the transport equation to extend DiPerna--Lions' theory so that it can work in bounded domains,
and the existence and uniqueness of a weak solution together with a stability result were shown.
However, the Dirichlet boundary condition was not treated in this work because the velocity field was assumed to equal zero on the boundary.
In contrast to this work, we {impose} inhomogeneous Dirichlet boundary conditions to the transport equation, and the velocity field is allowed to be nonzero on the boundary.
Moreover, our approach allows the domain to be unbounded, and hence the setting is more general than that in \cite{Tenan2021}.
Scott and Pollock \cite{ScottPollock2022} established the existence and the uniqueness of a weak solution
to the transport equation of vector-valued densities with the inflow boundary condition in {bounded} Lipschitz domains satisfying the Bernard condition.
Roughly speaking, the Bernard condition is a condition that the boundary can be split into two parts,
and one of which is the inflow boundary and the other is the complementary part.
Moreover, the intersection of these two parts is required to be the finite union of $(d-2)$-dimensional Lipschitz surfaces.
Here, by the inflow boundary condition, we mean that the velocity field {$\bv{b}$} goes inside the domain on the boundary.
To show the well-posedness of the problem, they assumed that the velocity field is solenoidal and that
the boundary data can be extended to the whole domain so that it is in some admissible function space.
{These hypotheses} are more restrictive than our {problem setting},
and hence their results are not directly applicable to our problem.

A key step to prove \Theorem{PRDir} is checking an invariance property of a weak solution, which is {the} so-called \textit{relabeling lemma}:
\begin{lemma}[Relabeling lemma] \label{LIn}
Suppose that $\Omega${, $h$ and $u_0$ satisfy} the hypotheses of \Theorem{PRDir} { and $\bv{b}\in L^1(0,T;W^{1,p'}(\Omega;\mathbb{R}^d))$}.
Assume that $u$ is a weak solution to \eqref{ETr}, \eqref{EDir}, and \eqref{EInit}.
Then, $\theta(u)$ is a weak solution to \eqref{ETr}, \eqref{EDir}, and \eqref{EInit}
with the boundary condition {$\theta(u)=\theta(h)$} on $\partial\Omega\times(0,T)$ and the initial condition $\theta(u)(\cdot, 0) = \theta(u_0)$ in $\Omega$
whenever $\theta\in C^1(\mathbb{R})$ with $\theta'\in L^\infty(\mathbb{R})$.
\end{lemma}
A solution having such a property is called a renormalized solution.
Multiple solutions constructed by a method of convex integration are not renormalized solutions.
It breaks the chain rule of differentiation as clarified in \cite{Ambrosio2004} and also in \cite{Tsuruhashi2021}.

Admitting the above lemma,
it is easy and standard to show the uniqueness (\Theorem{PRDir}) provided that $\divergence{\bv{b}}$ is bounded.
We will give the proof in \Section{sec:exst}.
However, in contrast to DiPerna and Lions \cite{DiPernaLions1989}, we do not include the case $p = \infty$ in \Theorem{PRDir}
since we need to argue by duality (see e.g., \cite[\S 2.3]{GigaGiga2024}), and hence the proof is more involved.

Let us explain the strategy of the proof of \Lemma{LIn}.
As a starting point, we construct an approximate equation of \eqref{ETr}, \eqref{EDir}, and \eqref{EInit} in terms of mollification.
To this end, we mollify a Lebesgue function $u$ so that it is smooth and write it as $u_\eta$.
However, the standard Friedrichs mollification is not suitable for approximating the Dirichlet boundary condition.
Indeed, if the ambient space is the half space, then the approximation by the standard mollification yields the half value of the Dirichlet data on the boundary (see e.g., \cite[Lemma 3.12]{GigaGiga2024}).
To cope with this issue, we introduce a mollifier $\mollHalf{\eta}$ tailored to the half line based on the moving average (see \Figure{fig:sm} and \Figure{fig:hm}).

    \begin{figure}[H]
        \centering
		\begin{minipage}{0.48\textwidth}
        	\includegraphics[keepaspectratio,width=80mm]{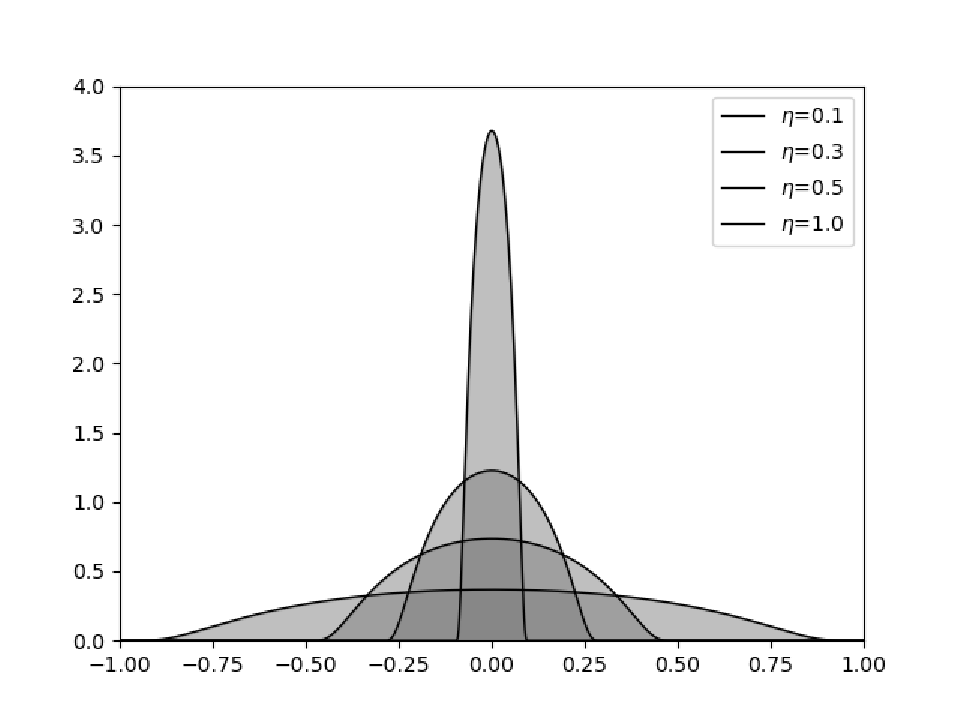}
        	\caption{The standard mollifier}\label{fig:sm}
		\end{minipage}
		\begin{minipage}{0.48\textwidth}
        	\includegraphics[keepaspectratio,width=80mm]{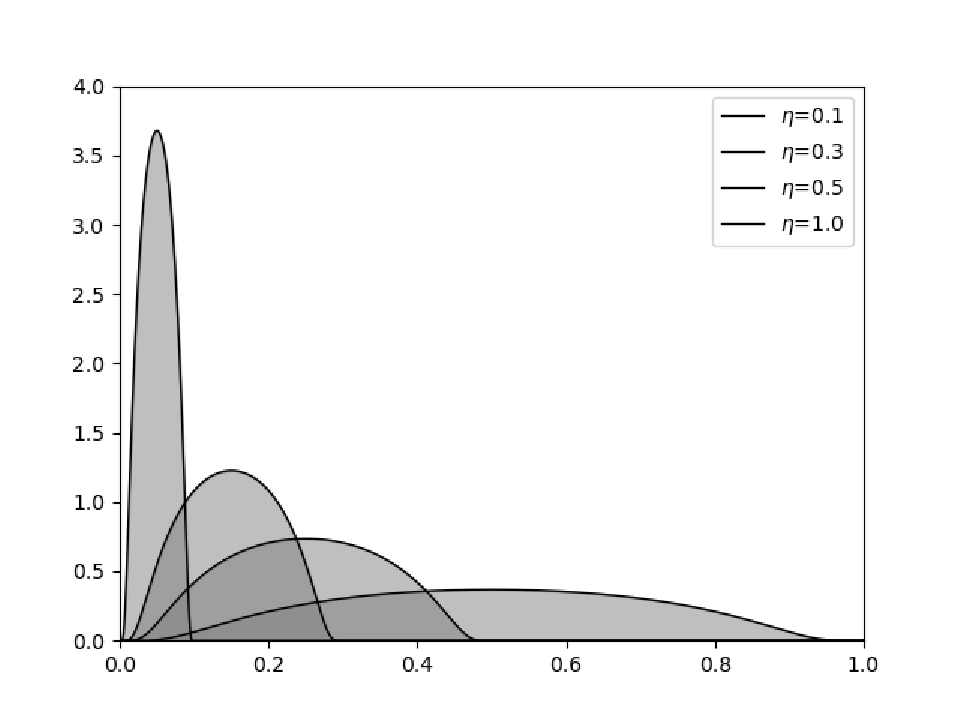}
        	\caption{The tailored mollifier}\label{fig:hm}
		\end{minipage}
    \end{figure}
Namely, while the solution $u$ is approximated by the standard mollification in the direction of the tangential space,
the Dirichlet boundary condition is approximated by the mollification in the direction of the half line.
The function $u_\eta$ should be smooth in the time variable as well, and hence this tailored mollification is also applied to the time variable.
Then, the approximate solution $u_\eta$ can be represented as follows:
\begin{equation}\label{eq:intro_1}
	u_\eta(x,t) = \int_0^\infty \conv{u(\cdot,s)}{\mollDHalf{d}{\eta}}(x)\mollHalf{\eta}(s-t)\,ds,
\end{equation}
where the convolution is operated with respect to the spatial variable $x\in\mathbb{R}^d$.
Though the integration with respect to the time variable can be rephrased as the convolution as well,
we make a point of writing the time variable explicitly to avoid confusion.
For more detail about the construction of the approximate solution, see \Section{sec:approx}.

Using the approximate solution $u_\eta$, we obtain the following approximate equation:
\begin{equation}\label{eq:intro_2}
	(u_\eta)_t + \bv{b}\cdot\nabla u_\eta = -\int_0^\infty \commutaor{u}{\bv{b}}(x,s)\mollHalf{\eta}(s-t)\,ds\qquad\mbox{in}\quad \Omega\times(0,T),
\end{equation}
where $\commutaor{u}{\bv{b}}$ is the residual term of the transport equation and is defined as follows:
\begin{equation}\label{eq:intro_3}
	\commutaor{u}{\bv{b}} := \convS{(\bv{b}\cdot\nabla u)}{\mollDHalf{d}{\eta}} - \bv{b}\cdot\nabla\conv{u}{\mollDHalf{d}{\eta}}.
\end{equation}
Here, {the gradient} $\nabla u$ should be understood as a distributional derivative of $u$ since $u$ is merely in $L^\infty(0,T;L^p(\Omega))$.
The quantity $\commutaor{u}{\bv{b}}$ is typically called the \textit{commutator} of $u$ and $\bv{b}$ in the literature.
To ensure that \eqref{eq:intro_2} approximates the equation \eqref{ETr},
we need to show that the commutator $\commutaor{u}{\bv{b}}$ converges to zero as $\eta\to 0$ in some topology at least in bulk.
For this purpose, we establish a DiPerna--Lions-type lemma for the commutator $\commutaor{u}{\bv{b}}$ (\Lemma{lem:comm}).
It is necessary to check whether the function $u_\eta$ is a good approximation of $u$ in the weak sense, say \Definition{DDir}.
In particular, we need to show that the approximate solution $u_\eta$ satisfies some approximate boundary conditions in the weak sense.
To this end, we invoke the definition of the weak solution and the approximate equation \eqref{eq:intro_2}
to derive an explicit form of the approximate boundary condition. Likewise, an approximate initial condition is also derived.
In the course of the proof, {it} is crucial {to choose test functions which are independent of the normal direction to the boundary $\partial\Omega$} (\Lemma{LMoll} {and \Remark{rem:spec}}).
We {shall first discuss} the half-space cases and then extend the argument to the general cases, namely $\Omega$ is assumed to be a regular domain.
In this case, the regularity of the boundary is important to ensure that the previous argument works well in a tubular neighborhood of the boundary.
In other words, we utilize the normal coordinate system near the boundary and obtain the same result as in the half-space cases (\Lemma{LMoll2}).
It remains to show the convergence of the commutator $\commutaor{u}{\bv{b}}$ as $\eta\to 0$.
Since we are working in the general domain, we need to consider the boundary effect.
However, the tailored mollification $\mollDHalf{d}{\eta}$ allows us to neglect the boundary integral term in the commutator thanks to its definition (see \Section{sec:comm}).
Therefore, we can follow the proof of \cite[Lemma 2.6]{GigaGiga2024} (\Lemma{lem:comm}).
Thanks to the explicit forms of the approximate boundary condition and the approximate initial condition,
we have the strong convergence of the approximate boundary condition and the approximate initial condition of $\theta(u_\eta)$ for
any smooth relabeling function $\theta$ with bounded derivative. This argument is the proof of \Lemma{LIn} that we will describe in \Section{sec:rl}.

The argument {mentioned above} works only for solenoidal velocity fields $\bv{b}$ (except for the convergence of the commutator)
since we have invoked the interchanging property of the commutator (\Lemma{lem:inchg}) which is true only for solenoidal velocity fields.
In fact, it turns out that this assumption can be removed by extending the interchanging property of the commutator so that it applies to
general velocity fields $\bv{b}$ which are not necessarily solenoidal (\Lemma{lem:inchg2}).
Using this property, the strong convergences of the approximate boundary condition and the approximate initial condition are shown in \Lemma{LMoll}
can be proved for general velocity fields $\bv{b}$ (\Lemma{LMoll3}).

We will demonstrate the existence of a weak solution to \eqref{ETr}, \eqref{EDir}, and \eqref{EInit} {in two different sets of hypotheses and} strategies.
{In both sets, w}e approximate the velocity field $\bv{b}$, the Dirichlet data $h$ and the initial data $u_0$ using the mollifiers that we have introduced so far,
and we write these {approximate functions} as $\bv{b}_\eta$, $h_\eta$ and $u_{0,\eta}$.

The first approach is to carry out a Gronwall-type argument to find a weakly convergent subsequence from a sequence of classical solutions {(\Theorem{thm:exst_2})}.
To derive a Gronwall-type inequality for the $L^p$ norm of a smooth approximate solution $u_\eta$,
we further assume that the velocity field $\bv{b}$ is essentially bounded {in addition to the hypotheses of the uniqueness result (\Theorem{PRDir})}.

The second approach is adopting the relabeling lemma (\Lemma{LIn}).
We renormalize the classical solution $u_\eta$ to \eqref{ETr}, \eqref{EDir} and \eqref{EInit} with $\bv{b} = \bv{b}_\eta$, $h = h_\eta$ and $u_0 = u_{0,\eta}$.
We choose a renormalizing function $\theta$ so that $\theta(u_\eta)$ is uniformly bounded in $L^\infty(0,T;L^p(\Omega))$ with respect to $\eta$.
Then, we extract a subsequence of $\theta(u_\eta)$ which is weakly convergent in $L^p(\Omega)$ topology.
The weak limit of this sequence will be pulled back to be identified as a weak solution to the original problem.
This strategy works {only} for smooth bounded domains $\Omega$, bounded Dirichlet data $h$ and bounded initial data $u_0$ ({\Corollary{cor:exst_1}}).

This paper is organized as follows.
In \Section{sec:approx}, we will introduce a mollification in the direction of the half space and will approximate a weak solution to \eqref{ETr}, \eqref{EDir} and \eqref{EInit} by this mollification.
In \Section{sec:rl}, we will prove the relabeling lemma (\Lemma{LIn}) using a commutator estimate (\Lemma{lem:comm}).
Here, we will {begin} with the case when the domain is the half space to give a clear explanation, and will discuss the general case.
In \Section{sec:dvgf}, we will extend the results so far to the case when the velocity field is not necessarily solenoidal.
In \Section{sec:comm}, we will prove the convergence of the commutator (\Lemma{lem:comm}).
As a conclusion of this paper, we show the uniqueness of a weak solution (\Theorem{PRDir}) using the relabeling lemma, and we demonstrate
two types of existence results of a weak solution (\Theorem{thm:exst_2} and {\Corollary{cor:exst_1}}) in \Section{sec:exst}.

\section{Approximation by mollifiers}\label{sec:approx}
	We let $\rho(x)$ be a function which satisfies
	\begin{equation*}
		\int_{\mathbb{R}}\rho(x)\,dx = 1,\quad\rho(x)\geq 0,\quad\rho(x) = \rho(-x),\quad\mbox{and}\quad \operatorname{supp}\rho\,\mbox{ is compact},
	\end{equation*}
	where $\operatorname{supp}\rho$ denotes the support of $\rho$, say the closure of the set where $\rho$ is nonzero.
	We define a standard mollifier $\rho_\eta$ in $\mathbb{R}$ by $\rho_\eta(x) := (1/\eta)\rho(x/\eta)$.
	Then, a mollifier $\rho_\eta^{(d-1)}$ in $\mathbb{R}^{d-1}$
	can be defined by $\rho^{(d-1)}_\eta(x) = \prod_{i=1}^{d-1}\rho_\eta(x_i)$ for $x = (x_1,\cdots,x_{d-1})$, recursively.
	Subsequently, we will approximate a weak solution to \eqref{ETr}, \eqref{EDir} and \eqref{EInit} in terms of mollifiers.
	To cope with the boundary condition properly, we introduce another mollifier replacing the standard mollifier
	for the Dirichlet boundary problems. Namely, let $\mollHalf{}\in C^\infty(\mathbb{R})$ be a function satisfying 
	\begin{equation*}
		\int_{0}^1\mollHalf{}(x)\,dx = 1,\quad\mollHalf{}(x)\geq 0,\quad \operatorname{supp}\mollHalf{} = [0,1].
	\end{equation*}
	For instance, a possible choice of $\mollHalf{}(x)$ is
	\begin{equation*}
		\mollHalf{}(x) := \begin{cases}
			C{\operatorname{exp}{\left(\frac{1}{4\left(x - \frac{1}{2}\right)^2 -1}\right)}}\qquad&\mbox{for}\quad x\in(0,1),\\
			0&\mbox{otherwise},
		\end{cases}
	\end{equation*}
	where a positive constant $C$ is selected so that the integration of $\mollHalf{}$ over $\mathbb{R}$ equals one.
	Then, we introduce a mollifier $\mollDHalf{d}{\eta}$ in $\mathbb{R}^d_+\,(:=\mathbb{R}^{d-1}\times\halfspace{}\quad\mbox{with}\quad\halfspace{} := (0,\infty))$ by $\mollDHalf{d}{\eta}(x) := \mollD{d-1}{\eta}(x')\mollHalf{\eta}(-x_d)$ for $x = (x',x_d)\in\mathbb{R}^d$
	with $\mollHalf{\eta}(x) := (1/\eta)\mollHalf{}(x/\eta)$ for $x\in\mathbb{R}$.
	See again \Figure{fig:hm} for the graph of $\mollHalf{\eta}(x)$ for several choices of $\eta > 0$.
	In particular, we set $\mollD{0}{}\equiv 1$, and hence $\mollDHalf{1}{} = \mollHalf{}$ for the case $d = 1$.

	We now define an approximate solution $u_\eta$ by
	\begin{equation}\label{def:moll}
		u_\eta(x, t) := \int_0^\infty\int_{\mathbb{R}^{d-1}}\int_{0}^\infty u(y',y_d,s)\mollD{d-1}{\eta}(x'-y')\mollHalf{\eta}(y_d-x_d)\mollHalf{\eta}(s-t)\,\dH{d-1}(y')dy_dds
	\end{equation}
	for $u\in L^p(\mathbb{R}^d_+\times(0,T))$, $x = (x',x_d)\in\mathbb{R}^d_+$ and $t\in(0,T)$.
	We note that $\mollHalf{\eta}(y_d - x_d) \equiv 0$ for $y_d \leq x_d$ and $\mollHalf{\eta}(s-t)\equiv 0$ for $s\leq t$ by its definition.
	Thus, the integral intervals of \eqref{def:moll} with respect to $y_d$ and $s$ make sense {without extending $u$ to $y_d < 0$ and $s < 0$}.
	\begin{rem}\label{rem:DPL}
		In \cite{DiPernaLions1989}, the function $u$ was not mollified in the time variable. However, we have noticed that
		the mollification in the time variable is necessary to guarantee that $u_\eta$ approximately solves the transport equation as stated in \Lemma{lem:approx} (see also \cite[Eq.(19)]{DiPernaLions1989}).
	\end{rem}
	In the sequel, we will write the convolution with respect to $x'\in\mathbb{R}^{d-1}$ and $x_d\in\mathbb{R}_+$ in $u_\eta$ as $u * \mollDHalf{d}{\eta}$,
	and then $u_\eta$ can be represented as
	\begin{equation}\label{def:moll2}
		u_\eta(x,t) = \int_0^\infty \conv{u(\cdot,s)}{\mollDHalf{d}{\eta}}(x)\mollHalf{\eta}(s-t)\,ds.
	\end{equation}
	Meanwhile, the convolutions $\convS{u}{\mollD{d-1}{\eta}}$ denotes the standard mollification in $\mathbb{R}^{d-1}$.
	Namely, we define
	\begin{equation*}
		\conv{u}{\mollD{d-1}{\eta}}(x) := \int_{\mathbb{R}^{d-1}}u(y)\mollD{d-1}{\eta}(x-y)\,\dH{d-1}(y),
	\end{equation*}
	where we use the convention that $\mollD{0}{\eta} = \moll{\eta}$.
	For later use, we also introduce a mollifier $\mollDTHalf{d}{\eta}$ with respect to $x'\in\mathbb{R}^{d-1}$ and $t\in\mathbb{R}$ defined by
	$\mollDTHalf{d}{\eta}(x',t) := \mollD{d-1}{\eta}(x')\mollHalf{\eta}(-t)$. Then, $u_\eta$ has another representation as follows (compare with \eqref{def:moll2}):
	\begin{equation}\label{def:moll3}
		u_\eta(x',x_d,t) = \int_0^\infty\conv{u(\cdot,x_d,\cdot)}{\mollDTHalf{d}{\eta}}(x',y_d,t)\mollHalf{\eta}(y_d-x_d)\,dy_d.
	\end{equation}

	As a conclusion of this section, we obtain an approximate transport equation \eqref{ETr} using the mollification $u_\eta$ defined by \eqref{def:moll2}:
	\begin{lemma}\label{lem:approx}
		Assume that {$\bv{b}\in L^1(0,T;W^{1,p'}(\halfspace{d};\mathbb{R}^d))$}.
		For $u\in L^\infty(0,T;L^p(\mathbb{R}^d_+))$, the function $u_\eta$ defined by \eqref{def:moll2} is a classical solution of
		\begin{equation}\label{eq:approx}
			(u_\eta)_t + \bv{b}\cdot\nabla u_\eta = -\int_0^\infty \commutaor{u}{\bv{b}}(x,s)\mollHalf{\eta}(s-t)\,ds\qquad\mbox{in}\quad \mathbb{R}^d_+\times(0,T),
		\end{equation}
		where
		\begin{equation}\label{def:comm}
			\commutaor{u}{\bv{b}} := \convS{(\bv{b}\cdot\nabla u)}{\mollDHalf{d}{\eta}} - \bv{b}\cdot\nabla\conv{u}{\mollDHalf{d}{\eta}}.
		\end{equation}
	\end{lemma}
	\begin{rem}\label{rem:grad_u}
		The formula \eqref{def:comm} includes the derivative of $u$, say $\nabla u$, and its meaning is not clear at first glance.
		This term is understood as a distributional derivative of $u$.
		{
			Namely, we define
			\begin{equation*}
				\left(\convS{\left(\bv{b}\cdot\nabla u\right)}{\mollDHalf{d}{\eta}}\right)(x) := - \int_{\halfspace{d}} u(y)\divergence_y{\left(\bv{b}(y)\mollDHalf{d}{\eta}(x-y)\right)\,dy}.
			\end{equation*}
		}
		For more detail, we will explain this in \Section{sec:comm}.
	\end{rem}
	\begin{proof}[\textbf{Proof of \Lemma{lem:approx}}]
		First, we note that the function $u_\eta$ is smooth in $\mathbb{R}^d_+\times(0,T)$ by its definition.
		We fix $(x,t) \in \mathbb{R}^d_+\times(0,T)$ and then take $\eta > 0$ so small that $\operatorname{supp}\mollDHalf{d}{\eta}(x-\cdot)\subset\subset\mathbb{R}^d_+$ and $\operatorname{supp}\mollHalf{\eta}(\cdot-t)\subset\subset(0,T)$.
		We implement a direct calculation to obtain
		\begin{equation}\label{eq:approx_1}
			\partial_t u_\eta(x,t) = -\int_0^\infty\int_{\halfspace{d}}u(y,s)\mollDHalf{d}{\eta}(x-y)(\mollHalf{\eta})'(s-t)\,dyds
		\end{equation}
		and
		\begin{multline}\label{eq:approx_2}
			(\bv{b}\cdot\nabla u_\eta)(x,t) = \int_0^\infty\left(\conv{\left(\bv{b}\cdot \nabla u\right)}{\mollDHalf{d}{\eta}}(x,s) - \commutaor{u}{\bv{b}}(x,s)\right)\mollHalf{\eta}(s-t)\,ds\\
			= -\int_0^\infty\int_{\halfspace{d}}u(y,s)\divergence_y{\left(\bv{b}(y,s)\mollDHalf{d}{\eta}(x-y)\right)}\mollHalf{\eta}(s-t)\,dyds - \int_0^\infty \commutaor{u}{\bv{b}}(x,s)\mollHalf{\eta}(s-t)\,ds\\
			= - \int_0^\infty\int_{\halfspace{d}}u(y,s)\left(\bv{b}(y,s)\cdot\nabla_y\mollDHalf{d}{\eta}(x-y) {+ \divergence{\bv{b}(y,s)\mollDHalf{d}{\eta}(x-y)}}\right)\mollHalf{\eta}(s-t)\,dyds\\
			- \int_0^\infty \commutaor{u}{\bv{b}}(x,s)\mollHalf{\eta}(s-t)\,ds.
		\end{multline}
		Here, we have invoked the distributional representation of $\convS{\left(\bv{b}\cdot\nabla u\right)}{\mollDHalf{d}{\eta}}$
		to obtain the second equality and the third equality, respectively.
		We add \eqref{eq:approx_1} and \eqref{eq:approx_2}, and take a test function as $\varphi(y,s) := \mollDHalf{d}{\eta}(x-y)\mollHalf{\eta}(s-t)$.
		Then, since $u$ is a weak solution to \eqref{ETr}, we obtain \eqref{eq:approx}.
		Here, we have used the fact that $\varphi\mid_{\mathbb{R}^{d-1}}\equiv 0$ and $\varphi(\cdot, 0)\equiv 0$ thanks to the choice of $\eta$.
	\end{proof}
	\begin{rem}\label{rem:comm}
		The quantity $\commutaor{u}{\bv{b}}$ is called the commutator of $u$ and $\bv{b}$.
		If $\bv{b}$ is a constant vector field, then $\commutaor{u}{\bv{b}}$ is trivially equal to zero.
		In fact, it turns out that $\commutaor{u}{\bv{b}}(\cdot,t)$ converges to zero in $L^p(\mathbb{R}^d_+)$ as $\eta\to 0$ for a.e. $t\in(0,T)$ even if $\bv{b}$ is not constant but merely in $L^1(0,T;W^{1,p'}(\mathbb{R}^d_+))$.
		We will prove this fact in \Section{sec:comm}.
		Therefore, we may regard \eqref{eq:approx} as an approximate transport equation \eqref{ETr}, and $\commutaor{u}{\bv{b}}$ can be regarded as a perturbation term due to this mollification.
	\end{rem}
	For the commutator $\commutaor{u}{\bv{b}}$, we will introduce an interchanging property that is useful for the discussion in \Section{sec:rl}:
	\begin{lemma}\label{lem:inchg}
		Assume that {$\bv{b}\in W^{1,p'}(\halfspace{d};\mathbb{R}^d)$ and} $\divergence{\bv{b}} = 0$. Then,
		for every $u\in L^p(\halfspace{d})$ {and $v\in L^\infty(\halfspace{d})$}, {we have}
		\begin{equation}\label{eq:inchg}
			\int_{\halfspace{d}}\commutaor{u}{\bv{b}}v\,dx = \int_{\halfspace{d}}\commutaor{v}{\bv{b}}u\,dx.
		\end{equation}
	\end{lemma}
	\begin{proof}
		The proof is straightforward by the integration by parts. Indeed, we compute
		\begin{align}\label{eq:inchg_1}
			&\int_{\halfspace{d}}\conv{(\bv{b}\cdot\nabla u)}{\mollDHalf{d}{\eta}}(x)v(x)\,dx = \int_{\halfspace{d}} \left(\int_{\halfspace{d}} (\bv{b}\cdot\nabla u)(y)\mollDHalf{d}{\eta}(x-y)\,dy\right)v(x)\,dx\nonumber\\
			&= \int_{\halfspace{d}} \left(- \int_{\halfspace{d}}u(y)\divergence_y{\left(\bv{b}(y)\mollDHalf{d}{\eta}(x-y)\right)}\,dy\right)v(x)\,dx\nonumber\\
			&= \int_{\halfspace{d}} \left(- \int_{\halfspace{d}}u(y)\bv{b}(y)\cdot\nabla_y\mollDHalf{d}{\eta}(x-y)\,dy\right)v(x)\,dx\nonumber\\
			&= \int_{\halfspace{d}} \left(-\int_{\halfspace{d}}u(y)\bv{b}(y)\cdot\nabla\mollDHalf{d}{\eta}(x-y)\,dy\right)v(x)\,dx\nonumber\\
			&= -\int_{\halfspace{d}} \bv{b}(y)\cdot\nabla\left(\int_{\halfspace{d}}v(x)\mollDHalf{d}{\eta}(y-x)\,dx\right)u(y)\,dy\nonumber\\
			&= -\int_{\halfspace{d}}\left(\bv{b}\cdot\nabla\conv{v}{\mollDHalf{d}{\eta}}\right)(y)u(y)\,dy,
		\end{align}
		where $\nabla_y$ denotes the gradient with respect to $y$.
		Here, we have invoked $\divergence{\bv{b}} = 0$
		to derive the third equality and the fourth equality, respectively;
		we have used $\nabla_y\mollDHalf{d}{\eta}(x-y) = \nabla\mollDHalf{d}{\eta}(x-y)$ to obtain the fourth equality on noting that
		$\mollDHalf{d}{\eta}(x-y) = \mollD{d-1}{\eta}(y'-x')\mollHalf{\eta}(y_d - x_d)$.
		Likewise, we obtain
		\begin{equation}\label{eq:inchg_2}
			\int_{\halfspace{d}}\conv{(\bv{b}\cdot\nabla v)}{\mollDHalf{d}{\eta}}(x)u(x)\,dx = -\int_{\halfspace{d}}\left(\bv{b}\cdot\nabla\conv{u}{\mollDHalf{d}{\eta}}\right)(y)v(y)\,dy.
		\end{equation}
		Subtracting \eqref{eq:inchg_2} from \eqref{eq:inchg_1} completes the proof.
	\end{proof}

\section{Relabeling lemma}\label{sec:rl}
	In this section, we prove the relabeling lemma (\Lemma{LIn}).
	For this, we {shall prove} a convergence result of a commutator.
	In \cite[Lemma 2.7]{GigaGiga2024}, we can find a corresponding result in the case when $\Omega := \mathbb{T}^d = \prod_{i=1}^d(\mathbb{R}/\omega_i\mathbb{Z})$ for some $\omega_i > 0\, (1\leq i\leq d)$,
	{which has no boundary}.
	For a while, we will postpone the proof of the following lemma and admit this.
	This kind of lemma is typically called either the Friedrichs lemma or the Diperna--Lions lemma (see e.g., \cite[Lemma II.1]{DiPernaLions1989} and \cite[Lemma 3.1]{BlouzaDret2001}).
	\begin{lemma}\label{lem:comm}
		Let $u\in L^p(\Omega)$ and $\bv{b}\in W^{1,\beta}(\Omega;\mathbb{R}^d)$ with $\beta\geq p'$. Let $\alpha \geq 1$ be such that $1/\alpha = 1/\beta + 1/p$.
		Then, {we have}
		\begin{equation}\label{eq:comm}
			\|r_\eta(u,\bv{b})\|_{L^\alpha(\Omega)} \leq C\|u\|_{L^p(\Omega)}\|\nabla\bv{b}\|_{L^\beta(\Omega)},
		\end{equation}
		where $\commutaor{u}{\bv{b}}$ is defined by \eqref{def:comm},
		and $C$ is a positive constant independent of $\eta$;
		{t}he symbol $\nabla\bv{b}$ denotes the matrix defined by $(\nabla\bv{b})_{i,j} := \partial_{x_j}b_i$ for $1\leq i,j\leq d${.}
		For a matrix $X$ in $\mathbb{R}^{d\times d}$,  $\|X\|$ denotes the Euclidean norm in $\mathbb{R}^{d\times d}$.
		In particular, the commutator $r_\eta(u,\bv{b})$ converges to zero in $L^\alpha(\Omega)$ as $\eta\to 0$.
	\end{lemma}
\begin{rem}
	We can find a similar convergence result for a commutator in a bounded domain $\Omega$ in $\mathbb{R}^d$ in the work by Blouza and Dret \cite{BlouzaDret2001}.
	Therein, the domain $\Omega$ was assumed to be merely Lipschitz which was weaker than our assumption on $\Omega$.
	They used the uniform cone condition {on $\Omega$} to define a convolution which is a candidate of an approximate solution:
	\begin{equation}\label{eq:cone}
		{\widetilde{u_\varepsilon}(x)} := \convDret{u}{\mollD{d}{\eta(\varepsilon)}}{\varepsilon}{\bv{e}}{(x)}
		:= \int_{\mathbb{R}^d}u(y)\mollD{d}{\eta(\varepsilon)}(x-\varepsilon\bv{e} -y)\,dy\qquad\mbox{with}\quad \eta(\varepsilon) := \varepsilon\sin\left(\frac{\theta_C}{2}\right),
	\end{equation}
	where $\bv{e}$ and $\theta_C$ are respectively a unit vector and an angle that determine the cone.
	Here, $\varepsilon$ should satisfy $0<\varepsilon < h_C/(1 + \sin(\theta_C/2))$ with $h_C$ being the height of the cone.
	Then, it can be seen that $u_\varepsilon$ is well-defined up to the boundary.
	In contrast to the standard mollification, their mollification was implemented in inscribed balls of the cone which are away from the boundary.
	See the support of $\mollD{d}{\eta(\varepsilon)}$ in \Figure{fig:bdsupport}:

    \begin{figure}[H]
        \centering
		\includegraphics[keepaspectratio,width=80mm]{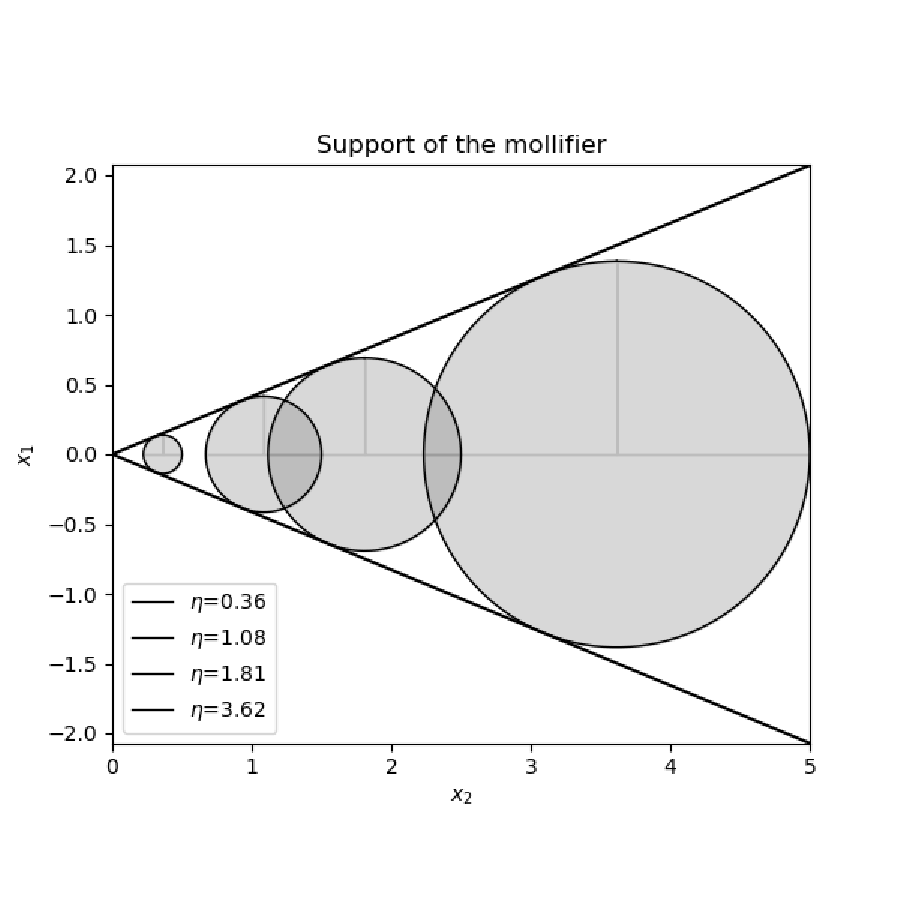}
		\caption{The cone and the support of $\mollD{d}{\eta(\varepsilon)}(\bv{0})$}
		\label{fig:bdsupport}
    \end{figure}

	We warn the reader that the argument{s} in \Section{sec:rl} {and \Section{sec:dvgf} do} not work if we replace $u_\eta$ by {$\widetilde{u_\varepsilon}$} since 
	{the convolution $\convDret{\cdot}{\cdot}{\varepsilon}{\bv{e}}$} does not {satisfy} the interchanging property \eqref{eq:inchg}.
	Indeed, we easily see that
	\begin{equation*}
		\int_{\mathbb{R}^d_+}\convDret{(\bv{b}\cdot\nabla u)}{\mollD{d}{\eta(\varepsilon)}}{\varepsilon}{\bv{e}}v\,dx = -\int_{\mathbb{R}^d_+}\bv{b}\cdot\nabla\left(\convDret{v}{\mollD{d}{\eta(\varepsilon)}}{\varepsilon}{-\bv{e}}\right)u\,dx.
	\end{equation*}
	Hence, defining the commutator $\widetilde{r_{{\varepsilon}}}(u,\bv{b})$ by
	\begin{equation*}
		\widetilde{r_\varepsilon}(u,\bv{b}) := \convDret{(\bv{b}\cdot\nabla u)}{\mollD{d}{\eta(\varepsilon)}}{\varepsilon}{\bv{e}} - \bv{b}\cdot\nabla\convDret{u}{\mollD{d}{\eta(\varepsilon)}}{\varepsilon}{-\bv{e}},	
	\end{equation*}
	we have the same consequence of \Lemma{lem:inchg} replacing $r_\eta$ by $\widetilde{r_\varepsilon}$.
	Extending the functions evenly across the boundary gives
	\begin{equation*}
		\bv{b}\cdot\nabla\left(\convDret{v}{\mollD{d}{\eta(\varepsilon)}}{\varepsilon}{-\bv{e}}\right)u = \bv{b}\cdot\nabla\left(\convDret{v}{\mollD{d}{\eta(\varepsilon)}}{\varepsilon}{\bv{e}}\right)u\qquad\mbox{on}\quad\mathbb{R}^{d-1}.
	\end{equation*}
	However, this is not the case for $x\in\mathbb{R}^d_+$.
\end{rem}
\begin{rem}[Regularity of velocity fields]\label{rem:reg}
	In \cite{BlouzaDret2001}, the regularity of the velocity field $\bv{b}$ was assumed to be stronger than ours, say $\bv{b}\in W^{1,\infty}(\Omega\times(0,T))$.
	However, in the light of our proof for \Lemma{lem:comm}, we may be able to relax the regularity assumption on $\bv{b}$, say $\bv{b}\in {L^1(0,T;W^{1,\beta}(\Omega;\mathbb{R}^d))}$ for $\beta \geq p'$.
\end{rem}
\begin{rem}\label{rem:bdry}
{Our proof for t}he main result of this paper cannot be extended to the case when $\Omega$ is Lipschitz
because we will use a normal coordinate system of $\partial\Omega$, {and this system is available when $\partial\Omega$ is much more regular}.
\end{rem}
\begin{rem}\label{rem:bdry2}
We cannot construct approximate solutions to the inhomogeneous initial-value problem \eqref{ETr}, \eqref{EDir} and \eqref{EInit} by $\widetilde{u}_\varepsilon$ defined by \eqref{eq:cone}
since the support of $\widetilde{u}_\varepsilon$ is strictly contained in $\Omega$ and is apart from the boundary $\partial\Omega$.
In other words, we have that $\widetilde{u}_\varepsilon\mid_{\partial\Omega}\equiv 0$ for every $0<\varepsilon\ll 1$.
However, we see that $\widetilde{u}_\varepsilon$ is still available for the homogeneous initial-boundary value problem, say the case $h\equiv 0$, since
the commutator for the convolution $\convDret{\cdot}{\cdot}{\varepsilon}{\mathbf{e}}$ was shown to converge to zero (see e.g., \cite[Lemma 3.1, Corollary 3.2]{BlouzaDret2001}),
and we need not invoke the interchanging property of the commutator.

In contrast to $\widetilde{u}_\varepsilon$,
our approximate solution $u_\eta$ defined by \eqref{def:moll} has non-trivial trace on $\partial\Omega$.
This fact makes it possible to construct approximate solutions to the inhomogeneous initial-boundary value problem \eqref{ETr}, \eqref{EDir}, and \eqref{EInit}.
\end{rem}
	We will give a proof of \Lemma{lem:comm} in \Section{sec:comm}.
	Admitting \Lemma{lem:comm}, we show the following lemma:
	\begin{lemma}\label{LMoll}
		Assume that $\Omega$ is the half space, say $\Omega = \mathbb{R}^{d}_+$, and $\partial\Omega = \mathbb{R}^{d-1}$.
		Suppose that $u$ is a weak solution of the initial-boundary value problem \eqref{ETr}, \eqref{EDir} and \eqref{EInit}.
		Assume that {$h$ and $u_0$ satisfy the hypotheses of \Theorem{PRDir},} $\bv{b}\in L^1(0,T;W^{1,p'}(\halfspace{d};\mathbb{R}^d))$ and $\divergence{\bv{b}} = 0$.
		Then, {we have}
		\begin{equation}\label{eq:LMoll}
			u_\eta(x', 0, t)(\bv{b}\cdot\nu_\Omega) = \conv{(h(\bv{b}\cdot\nu_\Omega))}{\mollDTHalf{d}{\eta}}(x',t)\qquad\mbox{for a.e. }(x',t)\in\mathbb{R}^{d-1}\times(0,T).
		\end{equation}
		Moreover, it {follows} that
		\begin{equation}\label{eq:LMollIni}
			u_\eta(x, 0) = \conv{u_0}{\mollDHalf{d}{\eta}}(x)\qquad\mbox{for a.e. }x\in\mathbb{R}^d_+.
		\end{equation}
		In particular, $u_\eta(\cdot,0,\cdot)(\bv{b}\cdot\nu_\Omega)$ converges to $h(\bv{b}\cdot\nu_\Omega)$ in $L^1(\mathbb{R}^{d-1}\times(0,T))$ as $\eta\to 0$, and
		$u_\eta(\cdot, 0)$ converges to $u_0$ in $L^p(\mathbb{R}^d_+)$ as $\eta\to 0$.
	\end{lemma}
	\begin{proof}
		For short notation, we abbreviate $\dL{d}(x)$ and $\dH{d-1}(x')$ by $dx$ and $dx'$, respectively unless confusion arises.
		For a test function $\varphi\in C^1_0(\closure{\mathbb{R}^d_+}\times(0,T))$, we now compute
		\begin{align}\label{eq:lem3_1}
			&\int_0^T\int_{\mathbb{R}^d_+}u_\eta\divergence{(\varphi\bv{b})}\,dxdt\nonumber\\
			&= \int_0^T\int_{\mathbb{R}^d_+}\left(\int_0^\infty \int_{\mathbb{R}^d_+} u(y,s)\mollDHalf{d}{\eta}(x-y)\mollHalf{\eta}(s-t)\,dyds\right)\divergence{(\varphi\bv{b})}(x,t)\,dxdt\nonumber\\
			&= \int_0^T\int_{\mathbb{R}^d_+}\left(\int_0^\infty \int_{\mathbb{R}^d_+} \left(\divergence{(\varphi\bv{b})}\right)(x,t)\mollDHalf{d}{\eta}(y-x)\mollHalf{\eta}(t-s)\,dxdt\right)u(y,s)\,dyds\nonumber\\
			&= \int_0^T\int_{\mathbb{R}^d_+}\left(\int_0^\infty \int_{\mathbb{R}^d_+} \left(\nabla\varphi\cdot\bv{b}\right)(x,t)\mollDHalf{d}{\eta}(y-x)\mollHalf{\eta}(t-s)\,dxdt\right)u(y,s)\,dyds\nonumber\\
			&= \int_0^T\int_{\mathbb{R}^d_+}\left(\int_0^\infty \conv{\left(\nabla\varphi\cdot\bv{b}\right)}{\mollDHalf{d}{\eta}}(y,t)\mollHalf{\eta}(t-s)\,dt\right)u(y,s)\,dyds\nonumber\\
			&= \int_0^T\int_{\mathbb{R}^d_+}\left\{\int_0^\infty \left(\bv{b}\cdot\nabla\conv{\varphi}{\mollDHalf{d}{\eta}}(y,t) + \commutaor{\varphi}{\bv{b}}(y,t)\right)\mollHalf{\eta}(t-s)\,dt\right\}u(y,s)\,dyds.
		\end{align}
		Letting
		\begin{equation}\label{def:Phi}
			\Phi(y,s) := \int_0^\infty \conv{\varphi}{\mollDHalf{d}{\eta}}(y,t)\mollHalf{\eta}(t-s)\,dt,
		\end{equation}
		{w}e have
		\begin{multline*}
			\int_0^T\int_{\halfspace{d}}u_\eta\divergence{(\varphi\bv{b})}\,dxdt \\= \int_0^T\int_{\halfspace{d}}(\bv{b}\cdot\nabla\Phi)\, u\,dxdt + \int_0^T\int_{\halfspace{d}}\left(\int_0^\infty\commutaor{\varphi}{\bv{b}}(x,s)\mollHalf{\eta}(s-t)\,ds\right)\,u\,dxdt.
		\end{multline*}
		We now choose $\varphi := \Phi$ in \Definition{DDir} as a test function to obtain
		\begin{multline*}
			-\int_0^T\int_{\halfspace{d}}u\partial_t\Phi\,dxdt - \int_{\halfspace{d}}u_0\Phi(\cdot,0)\,dx \\+ \int_0^T\int_{\mathbb{R}^{d-1}}h(\bv{b}\cdot\nu_\Omega)\Phi\biggm|_{\mathbb{R}^{d-1}}\,dx'dt - \int_0^T\int_{\halfspace{d}}(\bv{b}\cdot\nabla\Phi)u\,dxdt = 0,
		\end{multline*}
		and hence
		\begin{multline}\label{eq:lem3_4}
			\int_0^T\int_{\halfspace{d}}u_\eta\divergence{(\varphi\bv{b})}\,dxdt = \int_0^T\int_{\mathbb{R}^{d-1}}h(\bv{b}\cdot\nu_\Omega)\Phi\biggm|_{\mathbb{R}^{d-1}}\,dx'dt\\
			\quad - \int_0^T\int_{\halfspace{d}}u\,\partial_t\Phi\,dxdt - \int_{\halfspace{d}}u_0\Phi(\cdot,0)\,dx + \int_0^T\int_{\halfspace{d}}\left(\int_0^\infty r_\eta(\varphi,\bv{b})(x,s)\mollHalf{\eta}(s-t)\,ds\right)u\,dxdt.
		\end{multline}
		A direct calculation shows
		\begin{align*}
			&\int_0^T\int_{\mathbb{R}^d_+}u\,\partial_t\Phi\,dxdt = \int_0^T\int_{\halfspace{d}} u\partial_t\left(\int_0^\infty\conv{\varphi}{\mollDHalf{d}{\eta}}(x,s)\mollHalf{\eta}(s-t)\,ds\right)\,dxdt\\
			&= - \int_0^T\int_{\halfspace{d}}u(x,t)\left(\int_0^\infty\conv{\varphi}{\mollDHalf{d}{\eta}}(x,s)(\mollHalf{\eta})'(s-t)\,ds\right)\,dxdt\\
			&= - \int_0^T\int_{\halfspace{d}}u(x,t)\left\{\int_0^\infty\left(\int_{\halfspace{d}}\partial_t\varphi(y,s)\mollDHalf{d}{\eta}(x-y)\,dy\right)\mollHalf{\eta}(s-t)\,ds\right\}\,dxdt\\
			&= \int_0^T\int_{\halfspace{d}}\partial_t\varphi(y,s)\left\{\int_0^\infty\left(\int_{\halfspace{d}}u(x,t)\mollDHalf{d}{\eta}(y-x)\,dx\right)\mollHalf{\eta}(t-s)\,dt\right\}\,dyds\\
			&= \int_0^T\int_{\halfspace{d}} u_\eta\partial_t\varphi\,dyds.
		\end{align*}
		We now choose a test function $\widetilde{\varphi}\in C^1_0(\mathbb{R}^{d-1}\times[0,T))$ and {let $\delta > 0$. Then, we} define a function
		$\varphi\in C^1_0(\closure{\mathbb{R}^d_+}\times[0,T))$ by $\varphi(x,t) := \widetilde{\varphi}(x',t){\psi_\delta(x_d)}$ for each $x = (x',x_d)\in\mathbb{R}^d_+$ and $t\in(0,T)$,
		where $\psi_\delta : [0,\infty)\to [0,1]$ is a cut-off function such that
		\begin{equation*}
			\psi_\delta(x_d) = \begin{cases}
				1 & \mbox{if}\quad 0\leq x_d < \delta,\\
				0 & \mbox{if}\quad x_d > 2\delta.
			\end{cases}
		\end{equation*}
		Namely, $\varphi$ is constant in the $x_d$-direction {in $\mathbb{R}^{d-1}\times[0,\delta]\times(0,T)$}.
		Then, {we take $\eta > 0$ so small that $\operatorname{supp}\mollHalf{\eta}\subset[0,\delta]$ (see \Figure{fig:cutoff}).}

    \begin{figure}[H]
        \centering
		\includegraphics[keepaspectratio,width=80mm]{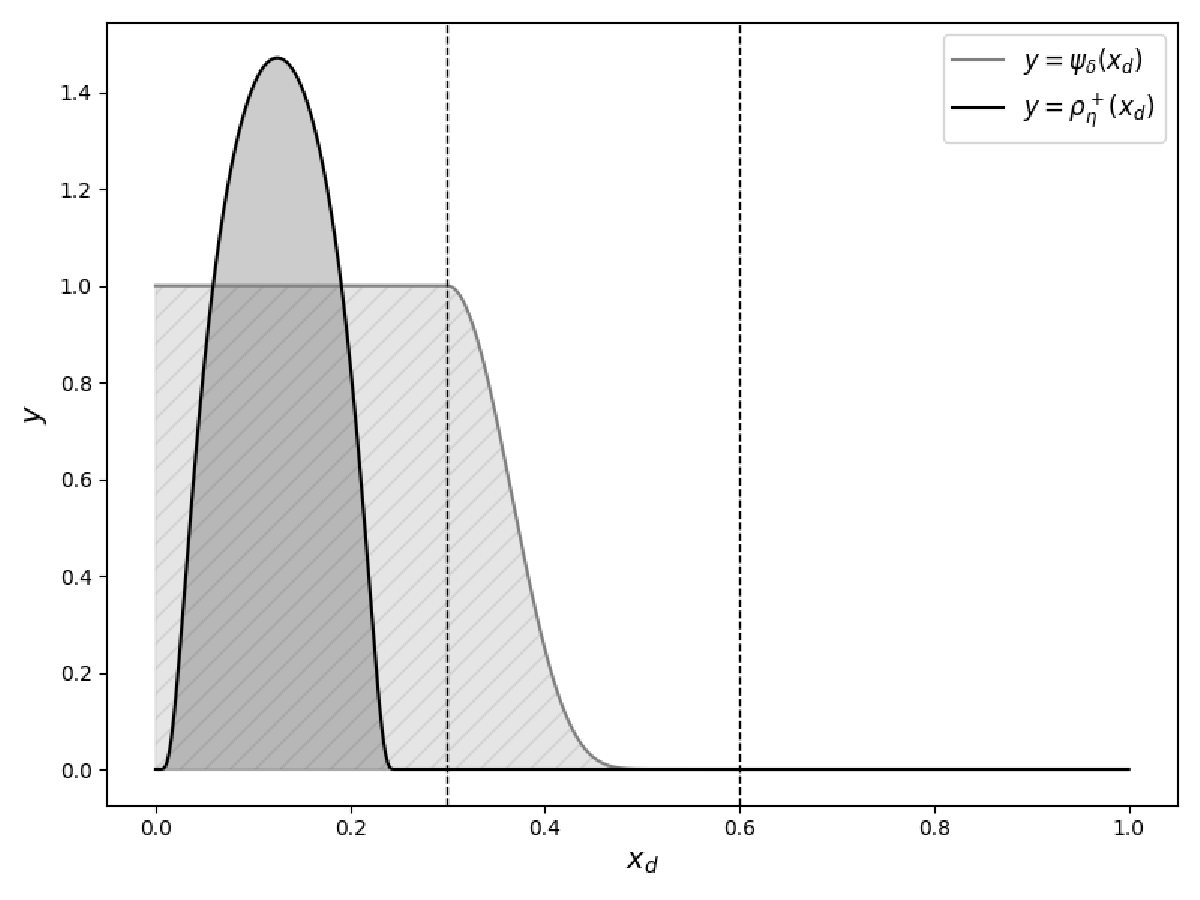}
		\caption{{The graphs of $\psi_\delta(x_d)$ and $\mollHalf{\eta}(x_d)$. Here, $\delta = 0.3$ and $\eta = 0.25$.}}
		\label{fig:cutoff}
    \end{figure}

		For this selection of $\varphi$, we can deform the right-hand side of \eqref{eq:lem3_4} including the boundary term as follows:
		\begin{align}\label{eq:lem3_5}
			&\int_0^T\int_{\mathbb{R}^{d-1}}h(\bv{b}\cdot\nu_\Omega)\Phi\biggm|_{\mathbb{R}^{d-1}}\,dx'dt\nonumber\\
			&= \int_0^T\int_{\mathbb{R}^{d-1}}\left(\int_0^\infty\conv{\varphi}{\mollDHalf{d}{\eta}}(x',0,s)\mollHalf{\eta}(s-t)\,ds\right)\,dx'dt\nonumber\\
			&= \int_0^T\int_{\mathbb{R}^{d-1}}h(\bv{b}\cdot\nu_\Omega)\left\{\int_0^\infty\left(\int_0^\infty\int_{\mathbb{R}^{d-1}}\varphi(y',0,s)\mollD{d-1}{\eta}(x'-y')\mollHalf{\eta}(y_d)\,dy'dy_d\right)\mollHalf{\eta}(s-t)\,ds\right\}\,dx'dt\nonumber\\
			&= \int_0^T\int_{\mathbb{R}^{d-1}}\widetilde{\varphi}(y',s)\left\{\int_0^\infty\int_{\mathbb{R}^{d-1}}h(\bv{b}\cdot\nu_\Omega)(x',t)\mollD{d-1}{\eta}(x'-y')\mollHalf{\eta}(t-s)\,dx'dt\right\}\,dy'ds\int_0^{{\delta}}\mollHalf{\eta}(y_d)\,dy_d\nonumber\\
			&= \int_0^T\int_{\mathbb{R}^{d-1}}\widetilde{\varphi}(y',s)\conv{(h(\bv{b}\cdot\nu_\Omega))}{\mollDTHalf{d}{\eta}}(y',s)\,dy'ds.
		\end{align}
		Here, we note that $\varphi(x',x_d, t) = \widetilde{\varphi}(x',t)$ {if $0\leq x_d \leq \delta$ and $\operatorname{supp}\mollHalf{\eta}\subset[0,\delta]$} to derive the third equality of \eqref{eq:lem3_5}.
		Meanwhile, as we have observed in \Section{sec:approx}, 
		the function $u_\eta$ is a classical solution to the following equation:
		\begin{equation}\label{eq:ueta}
			(u_\eta)_t + \bv{b}\cdot\nabla u_\eta = -\int_0^\infty r_\eta(u,\bv{b})(\cdot, s)\mollHalf{\eta}(s-t)\,ds\qquad\mbox{in}\quad\mathbb{R}^d_+\times(0,T).
		\end{equation}
		We test \eqref{eq:ueta} by $\varphi$ to obtain
		\begin{multline}\label{eq:lem3_7}
			-\int_0^T\int_{\mathbb{R}^d_+}u_\eta\partial_t\varphi\,dxdt - \int_0^T\int_{\mathbb{R}^d_+} u_\eta\divergence{(\varphi\bv{b})}\,dxdt 
			+ \int_0^T\int_{\mathbb{R}^{d-1}}(u_\eta(\bv{b}\cdot\nu_\Omega))(x',0,t)\widetilde{\varphi}(x',t)\,dx'dt \\
			- \int_{\halfspace{d}}u_\eta(x,0)\varphi(x,0)\,dx= -\int_0^T\int_{\mathbb{R}^d_+}\left(\int_0^\infty r_\eta(u,\bv{b})(x,s)\mollHalf{\eta}(s-t)\,ds\right)\varphi\,dxdt.
		\end{multline}
		Here, we note that the trace $u_\eta(\cdot,0,\cdot)$ is well-defined since $u_\eta$ is smooth in $\mathbb{R}^{d-1}$.
		We now recall the interchanging rule of the commutator (\Lemma{lem:inchg}) to deduce that
		\begin{multline*}
			\int_0^T\int_{\halfspace{d}}\left(\int_0^\infty \commutaor{u}{\bv{b}}(x,s)\mollHalf{\eta}(s-t)\,ds\right)\varphi\,dxdt\\
			= \int_0^T\int_{\halfspace{d}}\left(\int_0^\infty\commutaor{\varphi}{\bv{b}}(x,s)\mollHalf{\eta}(s-t)\,ds\right)u\,dxdt.
		\end{multline*}
		Combining \eqref{eq:lem3_4}, \eqref{eq:lem3_5} and \eqref{eq:lem3_7} yield
		\begin{equation*}
			\int_0^T\int_{\mathbb{R}^{d-1}}u_\eta(\bv{b}\cdot\nu_\Omega)\,\widetilde{\varphi}\,dx'dt = \int_0^T\int_{\mathbb{R}^{d-1}}\conv{(h(\bv{b}\cdot\nu_\Omega))}{\mollDTHalf{d}{\eta}}\widetilde{\varphi}\,dx'dt.
		\end{equation*}
		Since $\widetilde{\varphi}$ is arbitrary, we deduce from the fundamental lemma of calculus that
		$u_\eta(\bv{b}\cdot\nu_\Omega) = \convS{h(\bv{b}\cdot\nu_\Omega)}{\mollDTHalf{d}{\eta}}$ a.e. in $\mathbb{R}^{d-1}\times(0,T)$, and hence the formula \eqref{eq:LMoll} follows.

		The proof of \eqref{eq:LMollIni} is similar. We give it here for completeness. We choose a test function $\widetilde{\varphi}\in C^1_0(\mathbb{R}^d_+)$ and define $\varphi\in C^1_0(\closure{\mathbb{R}^d_+}\times(0,T))$ by
		$\varphi(x,t) := \widetilde{\varphi}(x)$ for each $x\in\mathbb{R}^d_+$, namely $\varphi$ is constant in time.
		Then, we observe that
		\begin{align}\label{eq:lem3_8}
			\int_{\halfspace{d}}u_0\Phi(\cdot, 0)\,dx &= \int_{\halfspace{d}}u_0(x)\int_0^\infty\conv{\varphi}{\mollDHalf{d}{\eta}}(y,0)\mollHalf{\eta}(t)\,dt\,dx\nonumber\\
			&= \int_{\halfspace{d}}u_0(x)\left(\int_{\halfspace{d}}\widetilde{\varphi}(y)\mollDHalf{d}{\eta}(x-y)\,dy\right)\,dx\nonumber\\
			&= \int_{\halfspace{d}}\widetilde{\varphi}(y)\left(\int_{\halfspace{d}}u_0(x)\mollDHalf{d}{\eta}(y-x)\,dx\right)\,dy\nonumber\\
			&= \int_{\halfspace{d}}\widetilde{\varphi}(y)\conv{u_0}{\mollDHalf{d}{\eta}}(y)\,dy.
		\end{align}
		Since $\operatorname{supp}\varphi\subset\subset\halfspace{d}$, the boundary integral terms of \eqref{eq:lem3_4} and \eqref{eq:lem3_7} vanish.
		Hence, we deduce from \eqref{eq:lem3_8} that
		\begin{equation*}
			\int_{\halfspace{d}}\conv{u_0}{\mollDHalf{d}{\eta}}\widetilde{\varphi}\,dx = \int_{\halfspace{d}}u_\eta(\cdot,0)\widetilde{\varphi}\,dx.
		\end{equation*}
		Since $\widetilde{\varphi}$ is arbitrary, we have the formula \eqref{eq:LMollIni}. The convergence of $u_\eta(\cdot,0)$ to $u_0$ in $L^p(\mathbb{R}^d_+)$ as $\eta\to 0$ follows from this {formula}. {The proof is now complete.}
	\end{proof}
	\begin{rem}\label{rem:spec}
		We emphasize that it is crucial to choose {$x_d$-independent }functions to derive the formulae for the boundary data \eqref{eq:LMoll} and the initial data \eqref{eq:LMollIni}.
	\end{rem}
	In \Lemma{LMoll}, we have only given a proof for the half spaces $\mathbb{R}^d_+$ to present the main idea.
	We shall discuss an extension of the consequence of \Lemma{LMoll} to general domains $\Omega$.
	In this case, the argument is more involved than the half-space cases due to the curved boundary $\partial\Omega$.
	However, we can still obtain a similar consequence of \Lemma{LMoll} even for general domains $\Omega$ having a regular boundary.
	\begin{lemma}\label{LMoll2}
		Assume that $\Omega$, $h$, and $u_0$ satisfy the hypotheses in \Theorem{PRDir},
		and that \linebreak {$\bv{b}\in L^1(0,T;W^{1,p'}(\Omega;\mathbb{R}^d))$ and} $\divergence{\bv{b}} = 0$.
		Then, {we have}
		\begin{equation*}
			u_\eta(\bv{b}\cdot\nu_\Omega)\to h(\bv{b}\cdot\nu_\Omega)\quad\mbox{in}\quad L^1(\partial\Omega\times(0,T))\quad\mbox{as}\quad\eta\to 0
		\end{equation*}
		and
		\begin{equation*}
			u_\eta(\cdot,0)\to \convS{u_0}{\mollDHalf{d}{\eta}}\quad\mbox{in}\quad L^p(\Omega)\quad\mbox{as}\quad\eta\to 0.
		\end{equation*}
	\end{lemma}
	\begin{proof}
		The basic idea of the proof is working on the boundary part $\partial\Omega$ of the domain $\Omega$
		and its normal direction separately near the boundary $\partial\Omega$.
		For each $\delta > 0$, we define a subdomain $\Omega_\delta$ of $\Omega$ by $\Omega_\delta := \{x\in\Omega\mid\operatorname{dist}(x,\partial\Omega) < \delta\}$.
		From the assumption on the regularity of $\partial\Omega$, the domain $\Omega$ satisfies the interior sphere condition.
		Thus, we can take $\delta > 0$ such that for every $x\in\Omega_{2\delta}$, there exists a unique point $\pi(x)\in\partial\Omega$ such that $|x-\pi(x)| = \operatorname{dist}(x,\partial\Omega)$.
		This relation defines the $C^1$ change of variables $\chgval:\Omega_{2\delta}\ni x\mapsto (\pi(x),-d(x))\in\partial\Omega\times(0,2\delta)$,
		where $d(x)$ denotes the signed distance function from $\partial\Omega$.

		We take any $\widetilde{\varphi}\in C^\infty_0(\partial\Omega\times(0,T))$ and try to show that
		\begin{equation}\label{eq:LMoll2_1}
			\int_0^T\int_{\partial\Omega}u_\eta(\bv{b}\cdot\nu_\Omega)\widetilde{\varphi}\,\dH{d-1}dt = \int_0^T\int_{\partial\Omega}\conv{h(\bv{b}\cdot\nu_\Omega)}{\mollDTHalf{d}{\eta}}\widetilde{\varphi}\,\dH{d-1}dt,
		\end{equation}
		where $u_\eta$ is now defined through the change of variables $\chgval$. If $x = (x',x_d) \in \Omega_{2\delta}$, then
		\begin{equation*}
			u_\eta(x',x_d,t) := \int_0^T\int_{\partial\Omega}\int_0^\infty u(y',y_d,s)\mollD{d-1}{\eta}(x'-y')\mollHalf{\eta}(y_d-x_d)\mollHalf{\eta}(s-t)J(\chgval(x))\,\dH{d-1}(y')dy_dds,
		\end{equation*}
		where $J(\chgval(x))$ denotes the Jacobian of the change of variables $\chgval$ at $x$, and its explicit form is given by
		\begin{equation}\label{eq:Jacobian}
			J(\chgval(x)) = \prod_{i=1}^{d-1}(1-\kappa_i(\pi(x))d(x))
		\end{equation}
		with $\kappa_i(\pi(x))\,(1\leq i\leq d-1)$ being the principal curvatures of $\partial\Omega$ at $\pi(x)$ (see \cite[Chapter 14, Appendix]{GilbargTrudinger1983}).
		We see that $J(\chgval(x))\neq 0$ holds in $\Omega_{2\delta}$ for sufficiently small $\delta > 0$ thanks to the interior sphere condition of $\partial\Omega$.
		Subsequently, we always choose such a positive constant $\delta$ unless otherwise mentioned.
		If $x\notin \Omega_{2\delta}$, then $u_\eta$ is defined by
		\begin{equation*}
			u_\eta(x',x_d,t) := \int_0^T\int_{\Omega}u(y,s)\mollD{d}{\eta}(x - y)\mollHalf{\eta}(s-t)\,\dL{d}(y)ds.
		\end{equation*}
		As in the case of the half space, we still use the convolution forms \eqref{def:moll2} and \eqref{def:moll3}.
		By the same argument as in \eqref{eq:lem3_1}, we observe that
		\begin{multline*}
			\int_0^T\int_{\Omega}u_\eta\divergence{(\varphi\bv{b})}\,\dL{d}dt =\\\int_0^T\int_{\Omega}u(\bv{b}\cdot\nabla\Phi)\,\dL{d}dt + \int_0^T\int_{\Omega}\left(\int_0^\infty\commutaor{\varphi}{\bv{b}}(\cdot,s)\mollHalf{\eta}(s-t)\,ds\right)\,u\,\dL{d}dt
		\end{multline*}
		for any test function $\varphi\in C^1_0(\closure{\Omega}\times(0,T))$, where $\Phi$ is defined by \eqref{def:Phi}.
		Since $u$ is a weak solution to \eqref{ETr}, \eqref{EDir} and \eqref{EInit}, we continue
		\begin{align}\label{eq:crv_1}
			\begin{split}
				&\int_0^T\int_{\Omega}u_\eta\divergence{(\varphi\bv{b})}\,\dL{d}dt = -\int_0^T\int_{\Omega}u_\eta\partial_t\Phi\,\dL{d}dt - \int_\Omega u_0\Phi(\cdot,0)\,\dH{d-1}\\
				&+\int_0^T\int_{\partial\Omega} h(\bv{b}\cdot\nu_\Omega)\Phi\biggm|_{\partial\Omega}\,\dH{d-1}dt + \int_0^T\int_{\partial\Omega}\left(\int_0^\infty\commutaor{\varphi}{\bv{b}}(\cdot,s)\mollHalf{\eta}(s-t)\,ds\right)u\,\dL{d}dt
			\end{split}
		\end{align}
		taking $\Phi$ as a test function.

		To show \eqref{eq:LMoll2_1}, we deform the boundary integral term on the right-hand side of the above equality.
		To this end, we extend the domain of $\widetilde{\varphi}$ to whole $\closure{\Omega}\times(0,T)$ by
		\begin{equation}\label{eq:LMoll2_2}
			\varphi(x,t) := 
			\begin{cases}
				\frac{\widetilde{\varphi}(\pi(x), t)}{J(\chgval(x))}&\qquad\mbox{if}\quad -\delta < d(x) \leq 0,\\
				0&\qquad\mbox{if}\quad -2\delta \geq d(x).
			\end{cases}
		\end{equation}
		For points $x\in \Omega$ such that $-2\delta < d(x) < -\delta$, we define $\varphi(x,t)$ in terms of some smooth cutoff function.
		The function $\varphi(x,t)$ defined in \eqref{eq:LMoll2_2} is well-defined and a $C^1$ function which is eligible for testing the equation since $J(\chgval(x))$ is $C^1$ (note that functions $\kappa_i(\pi(x))\,(1\leq i\leq d-1)$ are $C^1$ and $d(x)$ is $C^3$).
		Since $\operatorname{supp}{\mollHalf{\eta}} = [0,\eta]$, we can take $\eta$ so small that $\mollHalf{\eta}\equiv 0$ if $x$ is outside $\Omega_{\delta}$.
		For such $\eta > 0$, we compute
		\begin{align}\label{eq:crv_2}
			&\int_0^T\int_{\partial\Omega}h(\bv{b}\cdot\nu_\Omega)\Phi\biggm|_{\partial\Omega}\,\dH{d-1}dt\nonumber \\
			&= \int_0^T\,dt\int_{\partial\Omega}\,\dH{d-1}(y')h(\bv{b}\cdot\nu_\Omega)(y')\nonumber\\
			&\qquad\times\int_{\partial\Omega}\,\dH{d-1}(z')\int_0^\infty\,dz_d\,\varphi(z',z_d)\mollD{d-1}{\eta}(y'-z')\mollHalf{\eta}(z_d)J(\chgval(z',z_d))\nonumber \\
			&= \int_0^T\,dt\int_{\partial\Omega}\,\dH{d-1}(y')h(\bv{b}\cdot\nu_\Omega)(y')\nonumber\\
			&\qquad\times\int_{\partial\Omega}\,\dH{d-1}(z')\int_0^\delta\,dz_d\,\frac{\widetilde{\varphi}(z',t)}{J(\chgval(z',z_d))}\mollD{d-1}{\eta}(y'-z')\mollHalf{\eta}(z_d)J(\chgval(z',z_d))\nonumber \\
			&= \int_0^T\,dt\int_{\partial\Omega}\,\dH{d-1}(y')h(\bv{b}\cdot\nu_\Omega)(y')\int_{\partial\Omega}\,\dH{d-1}(z')\,\mollD{d-1}{\eta}(y'-z')\widetilde{\varphi}(z',t)\int_0^\delta\,dz_d\,\mollHalf{\eta}(z_d)\nonumber \\
			&= \int_0^T\,dt\int_{\partial\Omega}\,\dH{d-1}\,\conv{h(\bv{b}\cdot\nu_\Omega)}{\mollD{d-1}{\eta}}\widetilde{\varphi}.
		\end{align}
		Here, we have invoked that $\operatorname{supp}\mollHalf{\eta}\subset\subset\Omega_\delta$ to restrict the integral domain of $z_d$ to $[0,\delta]$ and 
		have used the definition of $\varphi$. Since $u_\eta$ is a classical solution (see \Lemma{lem:approx}) to
		\begin{equation*}
			(u_\eta)_t + \bv{b}\cdot\nabla u_\eta = -\int_0^\infty r_\eta(u, \bv{b})(\cdot,s)\mollHalf{\eta}(s-t)\,ds\qquad\mbox{in}\quad\Omega\times(0,T),
		\end{equation*}
		testing this equation by $\varphi$ gives
		\begin{multline}\label{eq:crv_3}
			\int_0^T\int_{\Omega}u_\eta\partial_t\varphi\,\dL{d}dt + \int_0^T\int_{\Omega}u_\eta\divergence{(\varphi\bv{b})}\,\dL{d}dt + \int_\Omega u_\eta(\cdot,0)\varphi(\cdot,0)\,\dL{d} \\
			- \int_0^T\int_{\partial\Omega}u_\eta(\bv{b}\cdot\nu_\Omega)\widetilde{\varphi}\,\dH{d-1}dt = \int_0^T\int_{\Omega}\left(\int_0^\infty r_\eta(u,\bv{b})(\cdot,s)\mollHalf{\eta}(s-t)\,ds\right)\varphi\,\dL{d}dt.
		\end{multline}
		Here, we note that $\varphi\mid_{\partial\Omega} = \widetilde{\varphi}$ since $d(x') = 0$ and $J(\chgval(x')) = 1$ hold for $x'\in\partial\Omega$.
		Combining \eqref{eq:crv_1}, \eqref{eq:crv_2}, and \eqref{eq:crv_3} together with the interchanging property \eqref{eq:inchg} of the commutator,
		we obtain \eqref{eq:LMoll2_1}, and hence it follows that
		\begin{equation*}
			u_\eta(\bv{b}\cdot\nu_\Omega) = \convS{(h(\bv{b}\cdot\nu_\Omega))}{\mollDTHalf{d}{\eta}}\qquad\mbox{a.e. in}\quad\partial\Omega\times(0,T).
		\end{equation*}
		Therefore, the convergence of $u_\eta(\cdot,0,\cdot)(\bv{b}\cdot\nu_\Omega)$ to $h(\bv{b}\cdot\nu_\Omega)$ in $L^p(\partial\Omega\times(0,T))$ follows.
		By a similar argument, it follows that
		\begin{equation*}
			\int_\Omega u_\eta(\cdot, 0)\widetilde{\varphi}\,\dL{d} = \int_\Omega u_0\widetilde{\varphi}\,\dL{d}\qquad\mbox{for all}\quad\widetilde{\varphi}\in C^1_0(\Omega).
		\end{equation*}
		Hence, $u_\eta(\cdot,0)\to u_0$ in $L^p(\Omega)$ also follows.
	\end{proof}
	\begin{rem}
		We have invoked the assumption that $\Omega$ has $C^3$ boundary to pull the test function $\varphi$ back to the boundary $\partial\Omega$ as $\widetilde{\varphi}$
		(see the last two equalities in \eqref{eq:crv_2}).
		We also need to take the parameter $\eta$ so small that $\varphi$ depends only on the tangential variables in $\partial\Omega$ inside the tubular domain $\Omega_\delta$.
		These are the main differences from the half-space cases.
		If the boundary is all flat, then we can take $\delta = 0$, and the proof is simplified due to $J(\chgval(x))\equiv 1$ as shown in \Lemma{LMoll}.
	\end{rem}
	We are now in the position to prove the relabeling lemma (\Lemma{LIn}).
	We assume that $\bv{b}$ is solenoidal in the following proof since we invoke the strong convergence result of the Dirichlet boundary data and the initial data of the approximate solution which has been shown in \Lemma{LMoll}.
	However, {the assumption that $\divergence{\bv{b}} = 0$ is not necessary for these convergence results as we will show in \Lemma{LMoll3}.
	We admit this at this stage and give the proof of the relabeling lemma:}
	\begin{proof}[\textbf{Proof of \Lemma{LIn}}]
		Since $\theta$ is smooth, the chain rule and \Lemma{lem:approx} imply that
		\begin{multline*}
			(\theta(u_\eta))_t + \mathbf{b}\cdot\nabla(\theta(u_\eta)) \\
			= \theta'(u_\eta)\left((u_\eta)_t + \mathbf{b}\cdot\nabla u_\eta\right) = -\theta'(u_\eta)\int_0^\infty r_\eta(u,\bv{b})(\cdot,s)\mollHalf{\eta}(s-t)\,ds\quad\mbox{in}\quad\mathbb{R}^d_+\times(0,T)
		\end{multline*}
		holds in the classical sense.
		Here, we have invoked the first equation of \eqref{eq:ueta} to derive the second equality.
		Testing the above equality by $\varphi\in C^\infty_0(\closure{\mathbb{R}^d_+}\times[0,T))$, we obtain that
		\begin{align}\label{eq:weak2}
			&-\int_0^T\int_{\mathbb{R}^d_+}\theta(u_\eta)\varphi_t\,\dL{d}dt - \int_{\mathbb{R}^d_+}\theta(u_\eta)(\cdot,0)\varphi(\cdot,0)\,\dL{d} + \int_0^T\int_{\mathbb{R}^{d-1}}(\theta(u_\eta)(\mathbf{b}\cdot\nu_\Omega)\varphi)(\cdot,0,\cdot)\,\dH{d-1}\nonumber \\
			&\quad - \int_0^T\int_{\mathbb{R}^d_+}\theta(u_\eta)\divergence{(\varphi\mathbf{b})}\,\dL{d}dt = \int_0^T\int_{\mathbb{R}^d_+}\theta'(u_\eta)\left(\int_0^\infty r_\eta(u,\bv{b})(\cdot,s)\mollHalf{\eta}(s-t)\,ds\right)\varphi\,\dL{d}dt.
		\end{align}
		Since $u_\eta(\cdot,0,\cdot)(\bv{b}\cdot\nu_\Omega)\to h(\bv{b}\cdot\nu_\Omega)$ in $L^1(\mathbb{R}^{d-1}\times(0,T))$ as $\eta\to 0$ and $\theta'\in L^\infty(\mathbb{R})$,
		we compute
		\begin{align*}
			&\left|\int_0^T\int_{\mathbb{R}^{d-1}}\theta(u_\eta)(\bv{b}\cdot\nu_\Omega)\varphi(\cdot,0,\cdot)\,\dH{d-1}dt - \int_0^T\int_{\mathbb{R}^{d-1}}\theta(h)(\bv{b}\cdot\nu_\Omega)\varphi(\cdot,0,\cdot)\,\dH{d-1}dt\right|	\\
			&{=\left|\int_0^T\int_{\mathbb{R}^{d-1}}\left(\int_0^1\theta'(su_\eta + (1-s)h)\,ds(u_\eta - h)\right)(\bv{b}\cdot\nu_\Omega)\varphi(\cdot,0,\cdot)\,\dH{d-1}dt\right|}\\
			&\leq \|\theta'\|_{L^\infty(\mathbb{R})}\sup_{K}|\varphi|\|u_\eta(\bv{b}\cdot\nu_\Omega) - h(\bv{b}\cdot\nu_\Omega)\|_{L^1(K)}\to 0 \quad\mbox{as}\quad\eta\to 0,
		\end{align*}
		where $K\subset\mathbb{R}^{d-1}\times(0,T)$ is the support of $\varphi$ which is compact.
		{Here, the convergence of the right-hand side is deduced from \Lemma{LMoll} and \Lemma{LMoll3} for the cases $\divergence{\bv{b}}\equiv 0$ and $\divergence{\bv{b}}\not\equiv 0$, respectively.}
		Likewise, we see that
		\begin{align*}
			&\left|\int_{\halfspace{d}}\theta(u_\eta)(\cdot,0)\varphi(\cdot,0)\,\dL{d} - \int_{\halfspace{d}}\theta(u_0)\varphi(\cdot,0)\,\dL{d}\right|\\
			&\leq {\left|\int_{\halfspace{d}}\left(\int_0^1\theta'(su_\eta + (1-s)u_0)\,ds(u_\eta - u_0)\right)\varphi(\cdot, 0)\,\dL{d}\right|}\\
			&\leq \|\theta'\|_{L^\infty(\mathbb{R})}\|\varphi(\cdot,0)\|_{L^\infty(\halfspace{d})}\|u_\eta(\cdot,0) - u_0\|_{L^1(\halfspace{d})}\to 0\qquad\mbox{as}\quad\eta\to 0,
		\end{align*}
		{where we have used the result for the initial data either in \Lemma{LMoll} or \Lemma{LMoll3} to obtain the last convergence.}
		Hence, letting $\eta\to 0$ in \eqref{eq:weak2} yields
		\begin{multline*}
			-\int_0^T\int_{\mathbb{R}^d_+}\theta(u)\varphi_t\,\dL{d}dt - \int_{\mathbb{R}^d_+}\theta(u_0)\varphi(\cdot,0)\,\dL{d} \\
			+ \int_0^T\int_{\mathbb{R}^{d-1}}\theta(h)(\mathbf{b}\cdot\nu_\Omega)\varphi\,\dH{d-1}dt
			-\int_0^T\int_{\mathbb{R}^d_+}\theta(u)\divergence{(\varphi\mathbf{b})}\,\dL{d}dt = 0.
		\end{multline*}
		This integral equation concludes the proof.
	\end{proof}
\section{Extension to non-divergence free vector fields}\label{sec:dvgf}
In this section, we extend the results of the previous sections to the case of more general vector fields $\bv{b}$.
Namely, we do not assume that $\divergence{\bv{b}} = 0$ and obtain the same results as stated in \Theorem{PRDir} and \Lemma{LIn}.
We begin with a generalization of the interchanging property of the commutator (\Lemma{lem:inchg}):

\begin{lemma}\label{lem:inchg2}
	Assume that $\bv{b}\in W^{1,{p'}}(\halfspace{d};\mathbb{R}^d)$.
	For every $u\in L^p(\halfspace{d})$ {and $v\in L^\infty(\halfspace{d})$}, it {follows} that	
	\begin{equation}\label{eq:inchg2}
		\int_{\halfspace{d}}\left(\commutaor{u}{\bv{b}} - \conv{u}{\mollDHalf{d}{\eta}}\divergence{\bv{b}}\right)v\,dx = \int_{\halfspace{d}}\left(\commutaor{v}{\bv{b}} - \conv{v}{\mollDHalf{d}{\eta}}\divergence{\bv{b}}\right)u\,dx.
	\end{equation}
\end{lemma}
\begin{rem}
	We note that the equation \eqref{eq:inchg2} is reduced to the interchanging property \eqref{eq:inchg} of the commutator for solenoidal vector fields $\bv{b}$.
\end{rem}
\begin{proof}[\textbf{Proof of \Lemma{eq:inchg2}}]
	{We invoke the distributional representation of $\nabla u$ and compute}
	\begin{multline*}
		\int_{\halfspace{d}}\conv{\left(\bv{b}\cdot\nabla u\right)}{\mollDHalf{d}{\eta}}v\,dx = -\int_{\halfspace{d}}\left(\int_{\halfspace{d}}u(y)\bv{b}(y)\cdot\nabla_y\mollDHalf{d}{\eta}(x-y)\,dy\right)v(x)\,dx\\
		-\int_{\halfspace{d}}\left(\int_{\halfspace{d}}(u\divergence{\bv{b}})(y)\mollDHalf{d}{\eta}(x-y)\,dy\right)v(x)\,dx.
	\end{multline*}
	We already know from the proof of \Lemma{lem:inchg} that the first term of the right-hand side can be computed as
	\begin{equation*}
		-\int_{\halfspace{d}}\left(\int_{\halfspace{d}}u(y)\bv{b}(y)\cdot\nabla_y\mollDHalf{d}{\eta}(x-y)\,dy\right)v(x)\,dx = -\int_{\halfspace{d}}\bv{b}\cdot\nabla\conv{v}{\mollDHalf{d}{\eta}}u\,dx.
	\end{equation*}
	Thus, we have
	\begin{equation}\label{eq:inchg3}
		\int_{\halfspace{d}}\conv{\left(\bv{b}\cdot\nabla u\right)}{\mollDHalf{d}{\eta}}v\,dx = -\int_{\halfspace{d}}\bv{b}\cdot\nabla\conv{v}{\mollDHalf{d}{\eta}}u\,dx - \int_{\halfspace{d}}u\divergence{\bv{b}}\conv{v}{\mollDHalf{d}{\eta}}\,dx.
	\end{equation}
	Likewise, we see that
	\begin{equation}\label{eq:inchg4}
		\int_{\halfspace{d}}\conv{\left(\bv{b}\cdot\nabla v\right)}{\mollDHalf{d}{\eta}}u\,dx = -\int_{\halfspace{d}}\bv{b}\cdot\nabla\conv{u}{\mollDHalf{d}{\eta}}v\,dx - \int_{\halfspace{d}}v\divergence{\bv{b}}\conv{u}{\mollDHalf{d}{\eta}}\,dx.
	\end{equation}
	We subtract \eqref{eq:inchg4} from \eqref{eq:inchg3} and obtain \eqref{eq:inchg2}.
\end{proof}
We shall prove the strong convergences of the trace $u_\eta$ on $\partial\halfspace{d}$ and $t = 0$ without assuming that $\bv{b}$ is solenoidal
in terms of \Lemma{lem:inchg2}:
\begin{lemma}\label{LMoll3}
	The consequence of \Lemma{LMoll} holds for every $\bv{b}\in W^{1,{p'}}(\halfspace{d};\mathbb{R}^d)$.
\end{lemma}
\begin{proof}
	The proof is almost the same as that of \Lemma{LMoll} except for the use of \Lemma{lem:inchg2} instead of \Lemma{lem:inchg}.
	First, we compute
	\begin{align}\label{eq:lem4_2_1}
		&\int_0^T\int_{\halfspace{d}}u_\eta\divergence{(\varphi\bv{b})}\,dxdt\nonumber\\
		&= \int_0^T\int_{\halfspace{d}}\left(\int_0^\infty\int_{\halfspace{d}}(\nabla\varphi\cdot\bv{b} + \varphi\divergence{\bv{b}})(x,t)\mollDHalf{d}{\eta}(y-x)\mollHalf{\eta}(t-s)\,dxdt\right)u(y,s)\,dyds\nonumber\\
		&= \int_0^T\int_{\halfspace{d}}\left(\int_0^\infty \left\{\convS{\left(\bv{b}\cdot\nabla\varphi\right)}{\mollDHalf{d}{\eta}} + \convS{\left(\varphi\divergence{\bv{b}}\right)}{\mollDHalf{d}{\eta}}\right\}(y,t)\mollHalf{\eta}(t-s)\,dt\right)u(y,s)\,dyds.
	\end{align}
	We have already computed the first term of the right-hand side in terms of the formula:
	\begin{equation*}
		\int_{\halfspace{d}}\conv{\left(\bv{b}\cdot\nabla\varphi\right)}{\mollDHalf{d}{\eta}}u\,dx = \int_{\halfspace{d}}\left(\bv{b}\cdot\nabla\conv{\varphi}{\mollDHalf{d}{\eta}} + \commutaor{\varphi}{\bv{b}}\right)u\,dx.
	\end{equation*}
	For the second term, we note that
	\begin{multline*}
		\int_{\halfspace{d}} \conv{\left(\varphi\divergence{\bv{b}}\right)}{\mollDHalf{d}{\eta}}u\,dy = \int_{\halfspace{d}}\left(\int_{\halfspace{d}}(\varphi\divergence{\bv{b}})(x)\mollDHalf{d}{\eta}(y-x)\,dx\right)u(y)\,dy\\
		= \int_{\halfspace{d}}\left(\int_{\halfspace{d}} u(y)\mollDHalf{d}{\eta}(x-y)\,dy\right)(\varphi\divergence{\bv{b}})(x)\,dx = \int_{\halfspace{d}}\conv{u}{\mollDHalf{d}{\eta}}\varphi\divergence{\bv{b}}\,dx.
	\end{multline*}
	Then, we apply \Lemma{lem:inchg2} for $u$ and $v := \varphi$ to obtain
	\begin{equation*}
		\int_{\halfspace{d}}\conv{u}{\mollDHalf{d}{\eta}}\varphi\divergence{\bv{b}}\,dx = - \int_{\halfspace{d}}\commutaor{\varphi}{\bv{b}}u\,dx + \int_{\halfspace{d}}\commutaor{u}{\bv{b}}\varphi\,dx + \int_{\halfspace{d}}\conv{\varphi}{\mollDHalf{d}{\eta}}u\divergence{\bv{b}}\,dx.
	\end{equation*}
	Hence, we see that
	\begin{multline}\label{eq:lem4_2_2}
		\int_0^T\int_{\halfspace{d}}u_\eta\divergence{(\varphi\bv{b})}\,dxdt\\
		= \int_0^T\int_{\halfspace{d}}\int_0^\infty\left(\commutaor{u}{\bv{b}}\varphi + \bv{b}\cdot\nabla\conv{\varphi}{\mollDHalf{d}{\eta}}u + \conv{\varphi}{\mollDHalf{d}{\eta}}u\divergence{\bv{b}}\right)\mollHalf{\eta}(s-t)\,ds\,dydt\\
		= \int_0^T\int_{\halfspace{d}}\left(\int_0^\infty\commutaor{u}{\bv{b}}\mollHalf{\eta}(s-t)\,ds\right)\varphi\,dxdt + \int_0^T\int_{\halfspace{d}}u\divergence{(\Phi\bv{b})}\,dxdt,
	\end{multline}
	where $\Phi$ is defined by \eqref{def:Phi}. We again choose $\varphi := \Phi$ in \Definition{DDir} and obtain
	\begin{multline}\label{eq:lem4_2_3}
		-\int_0^T\int_{\halfspace{d}}u\partial_t\Phi\,dxdt - \int_{\halfspace{d}}u_0\Phi(\cdot,0)\,dx\\
		+ \int_0^T\int_{\mathbb{R}^{d-1}} h(\bv{b}\cdot\nu_\Omega)\Phi\biggm|_{\mathbb{R}^{d-1}}\,dx'dt - \int_0^T\int_{\halfspace{d}}u\divergence{(\Phi\bv{b})}\,dxdt = 0.
	\end{multline}
	We combine \eqref{eq:lem4_2_3} with \eqref{eq:lem4_2_2} and obtain the same formula as \eqref{eq:lem3_4} with replacement of $(u,\varphi)$ by $(\varphi, u)$.
	Therefore, the remaining part of the proof is the same as that of \Lemma{LMoll}.
\end{proof}

\section{Convergence of commutators}\label{sec:comm}
In this section, we will prove {\Lemma{lem:comm} which asserts} the convergence of the commutator $r_\eta(u,\bv{b})\to 0$ as $\eta\to 0$.
For simplicity of the argument, we assume that $\Omega$ is the half space $\mathbb{R}^d_+$.
We {begin} with clarifying the meaning of the second term of $r_\eta(u,\bv{b})$. Indeed,
the gradient $\nabla u$ is not integrable if $u$ is merely in $L^p(\Omega\times(0,T))$.
Let us consider the case when $u$ is smooth in $\mathbb{R}^d_+\times(0,T)$.
Then, for $x = (x',x_d)\in\mathbb{R}^d_+$, the integration by parts gives
\begin{align*}
	&\conv{(\bv{b}\cdot\nabla u)}{\mollDHalf{d}{\eta}}(x',x_d) = \int_{\mathbb{R}^{d-1}}\int_0^\infty\,(\bv{b}\cdot\nabla u)(y',y_d)\mollD{d-1}{\eta}(x'-y')\mollHalf{\eta}(y_d-x_d)\,\dH{d-1}(y')dy_d\\
	&= \int_{\mathbb{R}^{d-1}}(u(\bv{b}\cdot\nu_\Omega))(y')\mollD{d-1}{\eta}(x'-y')\mollHalf{\eta}(-x_d)\,\dH{d-1}(y')\\
	&\quad - \int_{\mathbb{R}^{d-1}}\int_0^\infty u\,\divergence_y{\left(\bv{b}(y',y_d)\mollD{d-1}{\eta}(x'-y')\mollHalf{\eta}(y_d-x_d)\right)}\,\dH{d-1}(y')dy_d.
\end{align*}
Since $x_d\geq 0$, the first term of the right-hand side can be neglected due to $\mollHalf{\eta}(-x_d) = 0$.
Note that a similar argument will justify the neglect of the first term in the case of general domains $\Omega$.
Hence, we deduce that
\begin{equation*}
	\conv{(\bv{b}\cdot\nabla u)}{\mollDHalf{d}{\eta}}(x',x_d) = -\int_{\mathbb{R}^{d-1}}\int_0^\infty u\,\divergence_y{\left(\bv{b}(y',y_d)\mollDHalf{d}{\eta}(x - y)\right)}\,\dH{d-1}(y')dy_d.
\end{equation*}
We are now in the position to prove the convergence of the commutator $r_\eta(u,\bv{b})$.
\begin{proof}[\textbf{Proof of \Lemma{lem:comm}}]
	We follow the proof of \cite[Lemma 2.7]{GigaGiga2024}, although we need to modify the argument to fit our setting according to the change of mollification.
	For $x = (x',x_d)\in\mathbb{R}^d_+$, we have
	\begin{align*}
		&r_\eta(u,\bv{b}) = \conv{(\bv{b}\cdot\nabla u)}{\mollDHalf{d}{\eta}}(x',x_d) - \left(\bv{b}\cdot\nabla\conv{u}{\mollDHalf{d}{\eta}}\right)(x',x_d)\\
		&= -\int_{\mathbb{R}^{d-1}}\int_0^\infty u\,\divergence_y{\left(\bv{b}(y',y_d)\mollDHalf{d}{\eta}(x - y)\right)}\,\dH{d-1}(y')dy_d\\
		&\qquad -\bv{b}(x', x_d)\cdot\int_{\mathbb{R}^{d-1}}\int_0^\infty u(y',y_d)\nabla\mollDHalf{d}{\eta}(x - y)\,\dH{d-1}(y')dy_d\\
		&= \int_{\mathbb{R}^{d-1}}\int_0^\infty u(y',y_d)\{\bv{b}(y',y_d) - \bv{b}(x',x_d)\}\cdot\nabla\mollDHalf{d}{\eta}(x - y)\,\dH{d-1}(y') dy_d\\
		&\qquad - \int_{\mathbb{R}^{d-1}}\int_0^\infty (u\divergence{\bv{b}})(y',y_d)\mollDHalf{d}{\eta}(x - y)\,\dH{d-1}(y') dy_d.
	\end{align*}
	For short notation, we write the first term and the second term of the right-hand side of the above equality as $I(x)$ and $II(x)$, respectively.
	First, let us estimate the norm of $I$ in $L^\alpha(\mathbb{R}^d_+)$.
	Since $\operatorname{supp}\mollDHalf{d}{\eta}=\{(x',x_d)\in \closure{B(0,\eta)}\mid x_d \geq 0\}$, we deduce that
	\begin{align*}
		I(x) &= \int_{|x-y|\leq\eta}\{\bv{b}(y) - \bv{b}(x)\}\cdot\nabla\mollDHalf{d}{\eta}(x-y)u(y)\,\dL{d}(y)\\
		&= \int_{|x-y|\leq 1}\frac{\bv{b}(x + \eta z) - \bv{b}(x)}{\eta}\cdot\nabla\mollDHalf{d}{}(-z)u(x + \eta z)\,\dL{d}(z).
	\end{align*}
	Here, we have used the change of variables $z := (y - x) / \eta$.
	Setting $C_0 := \|\nabla\mollDHalf{d}{}\|_{L^\infty(\mathbb{R}^d_+)} < \infty$, we have
	\begin{equation*}
		|I(x)| \leq C_0\int_{|z|\leq 1}k_\eta(x,z)|u(x + \eta z)|\,\dL{d}(z),
	\end{equation*}
	where
	\begin{equation*}
		k_\eta(x,z) := \frac{|\bv{b}(x + \eta z) - \bv{b}(x)|}{\eta}.	
	\end{equation*}
	Thus, we obtain
	\begin{align}\label{eq:comm_1}
		\|I\|^\alpha_{L^\alpha(\mathbb{R}^d_+)}&\leq C^\alpha_0\int_{\mathbb{R}^d_+}\left\{\int_{|z|\leq 1}k_\eta(x,z)|u(x + \eta z)|\,\dL{d}(z)\right\}^\alpha\,\dL{d}(x)\nonumber\\
		&\leq C^\alpha_0 |B_1|^{\alpha-1}\int_{\mathbb{R}^d_+}\int_{|z|\leq 1}k_\eta(x,z)^\alpha|u(x + \eta z)|^\alpha\,\dL{d}(z)\,\dL{d}(x).
	\end{align}
	Here, we have invoked the Young inequality to derive the last inequality with $\alpha$ and $\alpha/(\alpha - 1)$, and
	$|B_1|$ denotes the $d$-dimensional Lebesgue measure of the unit ball $B(0,1)$.
	Let us write the right-hand side of the above inequality as $C^\alpha_0 J^\alpha$. Then, we deduce from the H\"{o}lder inequality
	for $k^\alpha_\eta(x,z)$ and $|w(x+\eta z)|^\alpha$ with $1/\alpha = 1/\beta + 1/p$ that
	\begin{align}\label{eq:comm_2}
		J &\leq |B_1|^{1-\frac{1}{\alpha}}\left\{\int_{\mathbb{R}^d_+}\int_{|z|\leq 1}k_\eta(x,z)^\beta\dL{d}(z)\dL{d}(x)\right\}^{\frac{1}{\beta}}\left\{\int_{\mathbb{R}^d_+}\int_{|z|\leq 1}|u(x+\eta z)|^p\,\dL{d}(z)\dL{d}(x)\right\}^{\frac{1}{p}}\nonumber\\
		&\leq |B_1|^{1-\frac{1}{\alpha}}\cdot\|u\|_{L^p(\mathbb{R}^d_+)}|B_1|^{\frac{1}{p}}\left\{\int_{\mathbb{R}^d_+}\int_{|z|\leq 1}k_\eta(x,z)^\beta\dL{d}(z)\dL{d}(x)\right\}^{\frac{1}{\beta}}\nonumber\\
		&= C_1\|w\|_{L^p(\mathbb{R}^d_+)}\left\{\int_{\mathbb{R}^d_+}\int_{|z|\leq 1}k_\eta(x,z)^\beta\dL{d}(z)\dL{d}(x)\right\}^{\frac{1}{\beta}},
	\end{align}
	where we set $C_1 := |B_1|^{1-1/\beta}$. Meanwhile, we deduce from the fundamental theorem of calculus that
	\begin{align*}
		|\bv{b}(x + \eta z) - \bv{b}(x)| &= \sqrt{\sum_{i=1}^d\left(\int_0^1\frac{d}{ds}b_i(x + \eta zs)\,ds\right)^2} = \sqrt{\sum_{i=1}^d\left(\int_0^1\nabla b_i(x + \eta zs)\cdot \eta z\,ds\right)^2}\\
		&\leq \eta|z|\int_0^1 \|\nabla\bv{b}(x + \eta zs)\|\,ds,
	\end{align*}
	where $b_i$ denotes the $i$-th element of $\bv{b}$.
	Hence, we see that
	\begin{align}\label{eq:comm_3}
		&\int_{\mathbb{R}^d_+}\int_{|z|\leq 1} k_\eta(x,z)^\beta\,\dL{d}(z)\,\dL{d}(x) \leq \int_{\mathbb{R}^d_+}\int_{|z|\leq 1}\left(\int_0^1\|\nabla\bv{b}(x + \eta zs)\|\,ds\right)^\beta|z|^\beta\,\dL{d}(z)\,\dL{d}(x)\nonumber\\
		&\leq \int_{\mathbb{R}^d_+}\int_{|z|\leq 1}\left(\int_0^1\|\nabla\bv{b}(x + \eta zs)\|^\beta\,ds\right)|z|^\beta\,\dL{d}(z)\,\dL{d}(x)\leq C^\beta_2\|\nabla\bv{b}\|^\beta_{L^\beta(\mathbb{R}^d_+)},
	\end{align}
	where we have set $C_2 := \left(\int_{|z|\leq 1}|z|^\beta\,\dL{d}(z)\right)^{\frac{1}{\beta}}$.
	Combining \eqref{eq:comm_1}, \eqref{eq:comm_2}, and \eqref{eq:comm_3}, we obtain
	\begin{equation}\label{eq:comm_4}
		\|I\|_{L^\alpha(\mathbb{R}^d_+)} \leq C_0 J \leq C_0C_1C_2\|u\|_{L^p(\mathbb{R}^d_+)}\|\nabla\bv{b}\|_{L^\beta(\mathbb{R}^d_+)}.
	\end{equation}
	It remains to estimate the norm of $II$ in $L^\alpha(\mathbb{R}^d_+)$, although the argument is straightforward from the H\"{o}lder inequality:
	\begin{equation}\label{eq:comm_5}
		\|II\|_{L^\alpha(\mathbb{R}^d_+)} \leq \|\mollDHalf{d}{\eta}\|_{L^1(\mathbb{R}^d_+)}\|u\,\divergence{\bv{b}}\|_{L^\alpha(\mathbb{R}^d_+)} =  \|u\,\divergence{\bv{b}}\|_{L^\alpha(\mathbb{R}^d_+)}\leq \|u\|_{L^p(\mathbb{R}^d_+)}\|\nabla\bv{b}\|_{L^\beta(\mathbb{R}^d_+)}.\\
	\end{equation}
	Here, the Young inequality has been invoked to derive the first inequality.
	Hence, \eqref{eq:comm} follows from \eqref{eq:comm_4} and \eqref{eq:comm_5} together with the Minkowski inequality.

	For the convergence of $r_\eta(u,\bv{b})$, we first consider the case when $u\in W^{1,p}(\mathbb{R}^d_+)$.
	In this case, both of $\convS{(\bv{b}\cdot\nabla u)}{\mollDHalf{d}{\eta}}$ and $\bv{b}\cdot\nabla\conv{u}{\mollDHalf{d}{\eta}}$ converge to $(\bv{b}\cdot\nabla)u$
	in $L^\alpha(\mathbb{R}^d_+)$ as $\eta\to 0$, and hence $r_\eta(u,\bv{b})\to 0$ in $L^\alpha(\mathbb{R}^d_+)$.
	Meanwhile, if $u$ is merely in $L^p(\mathbb{R}^d_+)$, we can approximate $u$ by a sequence $\{u_n\}_{n\in\mathbb{N}}$ of functions in $W^{1,p}(\mathbb{R}^d_+)$.
	Then, the convergence of $r_\eta(u,\bv{b})$ follows from the convergence of $r_\eta(u_n,\bv{b})$ for each $n$ and \eqref{eq:comm} as follows:
	\begin{align*}
		\|r_\eta(u,\bv{b})\|_{L^\alpha(\mathbb{R}^d_+)} &\leq \|r_\eta(u - u_n,\bv{b})\|_{L^\alpha(\mathbb{R}^d_+)} + \|r_\eta(u_n,\bv{b})\|_{L^\alpha(\mathbb{R}^d_+)}\\
		&\leq C\|u-u_n\|_{L^p(\mathbb{R}^d_+)}\|\nabla\bv{b}\|_{L^\beta(\mathbb{R}^d_+)} + \|r_\eta(u_n,\bv{b})\|_{L^\alpha(\mathbb{R}^d_+)}\to 0\qquad\mbox{as}\quad\eta\to 0.
	\end{align*}
	{The proof is now complete}.
\end{proof}
\begin{rem}
	In contrast to the proofs of \Lemma{lem:inchg} and \Lemma{LMoll},
	we have not used the assumption that $\bv{b}$ is divergence-free to prove the convergence of the commutator $r_\eta(u,\bv{b})$.
\end{rem}
\begin{rem}
	The proof of \Lemma{lem:comm} has been carried out without the time variable, and hence the convergence of the commutator $r_\eta(u,\bv{b})$ holds pointwise in time.
\end{rem}
\section{Existence and uniqueness of weak solution}\label{sec:exst}
In this section, we establish the existence and the uniqueness of a weak solution to the transport equation \eqref{ETr} with the Dirichlet data \eqref{EDir} and the initial data \eqref{EInit}.
Let us show the proof of \Theorem{PRDir} using the relabeling lemma (\Lemma{LIn}):
\begin{proof}[\textbf{Proof of \Theorem{PRDir}}]
Although we follow the argument in \cite[Lemma 2.5]{GigaGiga2024}, we need to take the boundary condition into account.
 We may assume that $u_0\equiv 0$, and $h=0$ on the place where $\bv{b}\cdot\nu_\Omega\leq 0$ on $\partial\Omega$.
 For each $M > 0$ and $\eta > 0$, we define functions $g_M(u)$ and $\theta_\eta(u)$ by
 \begin{align*}
	g_M(u) &:= (|u|\land M)^p,\\
	\theta_{\eta,M}(u) &:= (g_M * \rho_\eta)(u) - (g_M * \rho_\eta)(0)\qquad\mbox{for}\quad u\in\mathbb{R},
 \end{align*}
 where $a\land b := \min\{a,b\}$ for $a,\,b\in\mathbb{R}$.
 Here, the convolution used above means the standard mollified function.
 The function $\theta_{\eta,M}$ is clearly smooth. Moreover, its derivative is bounded in $\mathbb{R}$. Indeed, we compute
 \begin{equation*}
	\theta_{\eta,M}'(u) = \int_{\mathbb{R}}g_M'(u-y)\rho_\eta(y)\,dy = \int_{u-M}^{u+M}p|u-y|^{p-1}\rho_\eta(y)\,dy\leq pM^{p-1}.
 \end{equation*}
 Thus, we deduce from \Lemma{LIn} that $\theta_{\eta,M}(u)$ is a weak solution of \eqref{ETr} with the Dirichlet data $\theta_{\eta,M}(h)$
 and the initial data $\theta_{\eta,M}(u)(\cdot,0)\equiv 0$ since $\theta_{\eta,M}(0) = 0$.
 We take any $\psi(t)\in C^\infty_0(0,T)$ and then obtain
 \begin{equation}\label{eq:thm1_1}
	-\int_0^T\int_{\Omega}\theta_{\eta,M}(u)\psi_t\,\dL{d}dt + \int_0^T\int_{\partial\Omega}\theta_{\eta,M}(h)(\bv{b}\cdot\nu_\Omega)\psi\,\dH{d-1}dt -\int_0^T\int_\Omega \theta_{\eta,M}(u)\psi\divergence{\bv{b}}\,\dL{d}dt = 0.
 \end{equation}
 Here, we have invoked that $\psi(t)$ is constant in the spatial variable.
 We can deform \eqref{eq:thm1_1} as follows:
 \begin{equation*}
	\int_0^T\psi(t)\left\{\frac{d}{dt}\left(\int_\Omega\theta_{\eta,M}(u)\,\dL{d}\right) + \int_{\partial\Omega}\theta_{\eta,M}(h)(\bv{b}\cdot\nu_\Omega)\,\dH{d-1} - \int_\Omega\theta_{\eta,M}(u)\divergence{\bv{b}}\,\dL{d}\right\}\,dt = 0.
 \end{equation*}
 We note that the time derivative here should be understood in the sense of distribution.
 Since $\psi(t)$ is arbitrary, the fundamental lemma of the calculus of variations gives
 \begin{equation*}
	\frac{d}{dt}\int_{\Omega}\theta_{\eta,M}(u)(\cdot,t)\,\dL{d} + \int_{\partial\Omega}(\theta_{\eta,M}(h)(\bv{b}\cdot\nu_\Omega))(\cdot,t)\,\dH{d-1} - \int_\Omega\theta_{\eta,M}(u)\divergence{\bv{b}}\,\dL{d} = 0
 \end{equation*}
 for a.e. $t\in(0,T)$.
 Letting $\eta\to 0$ shows that
 \begin{equation*}
	\frac{d}{dt}\int_\Omega(|u(\cdot,t)|\land M)^p\,\dL{d} + \int_{\partial\Omega}(|h(\cdot,t)|\land M)^p(\bv{b}\cdot\nu_\Omega)\,\dH{d-1} - \int_\Omega({|u(\cdot,t)|}\land M)^p\divergence{\bv{b}}\,\dL{d} = 0
 \end{equation*}
 for a.e. $t\in(0,T)$.
 Since $h(\cdot,t)\equiv 0$ on the place where $\bv{b}\cdot\nu_\Omega< 0$, we have
 \begin{equation}\label{E:LIn1}
	\frac{d}{dt}\int_\Omega(|u(\cdot,t)|\land M)^p\,\dL{d} \leq \|\divergence{\bv{b}}\|_\infty\int_\Omega(|u(\cdot,t)|\land M)^p\,\dL{d}.
 \end{equation}
 We integrate \eqref{E:LIn1} over $(0,t)$ and obtain
 \begin{equation*}
	\int_\Omega(|u(\cdot,t)\land M|)^p\,\dL{d} \leq \int_\Omega(|u(\cdot, 0)|\land M)^p\,\dL{d} + \|\divergence{\bv{b}}\|_\infty \int_0^t\int_\Omega(|u(\cdot,s)|\land M)^p\,\dL{d}ds.
 \end{equation*}
 Thus, the Gronwall inequality yields
 \begin{equation*}
	\int_\Omega(|u(\cdot,t)|\land M)^p\,\dL{d} \leq \int_\Omega(|u(\cdot, 0)|\land M)^p\,\dL{d}e^{\|\divergence{\bv{b}}\|_\infty t}\qquad\mbox{for a.e.}\quad t\in(0,T).
 \end{equation*}
 Hence, the assumption that $u(\cdot,0) = u_0\equiv 0$ implies
 \begin{equation*}
	\int_\Omega(|u(\cdot,t)|\land M)^p\,\dL{d} = 0\qquad\mbox{for a.e.}\quad t\in(0,T).
 \end{equation*}
 Since $u\in L^\infty(0,T;L^p(\Omega))$, we deduce from the Lebesgue dominated convergence theorem that
 \begin{equation*}
	\int_\Omega|u(\cdot,t)|^p\,\dL{d} = 0\qquad\mbox{for a.e.}\quad t\in(0,T)
 \end{equation*}
 sending $M\to\infty$. This implies that $u\equiv 0$ a.e. in $\Omega\times(0,T)$.
\end{proof}
Subsequently, we demonstrate the existence of a weak solution to \eqref{ETr}, \eqref{EDir}, and \eqref{EInit}.
\begin{proof}[{\textbf{Proof of \Theorem{thm:exst_2}}}]
	For each $\eta > 0$, let $\bv{b}_\eta$, $h_\eta$ and $u_{0,\eta}$ be smooth approximations of $\bv{b}$, $h$ and $u_0$ in terms of the mollifiers.
	Precisely, we set
	\begin{equation*}
		\bv{b}_\eta(x,t) := \int_0^\infty\conv{\bv{b}}{\mollDHalf{d}{\eta}}(x,s)\mollHalf{\eta}(s-t)\,ds,\quad h_\eta := \convS{h}{\mollDTHalf{d}{\eta}},\quad\mbox{and}\quad u_{0,\eta} := \convS{u_0}{\mollDHalf{d}{\eta}}.
	\end{equation*}
	Then, there exists a classical solution $u_\eta$ to the following problem:
	\begin{equation}\label{eq:exst_3}
		\begin{cases}
			u_\eta + \bv{b}_\eta\cdot\nabla u_\eta = 0&\qquad\mbox{in}\quad\Omega\times(0,T),\\
			u_\eta = h_\eta&\qquad\mbox{on}\quad\partial\Omega\times(0,T),\\
			u_\eta(\cdot,0) = u_{0,\eta}&\qquad\mbox{in}\quad\Omega.
		\end{cases}
	\end{equation}
	For each $t\in(0,T)$, we compute
	\begin{equation}\label{eq:exst_4}
		\frac{d}{dt}\|u_\eta(\cdot,t)\|^p_p = p\int_\Omega |u_\eta(\cdot,t)|^{p-1}\operatorname{sign}(u_\eta)\partial_t u_\eta\,dx = -p \int_\Omega |u_\eta(\cdot,t)|^{p-1}\operatorname{sign}(u_\eta)\bv{b}_\eta\cdot\nabla u_\eta\,dx.
	\end{equation}
	Moreover, we implement the integration by parts to the right-hand side of the above equation to obtain
	\begin{align}\label{eq:exst_5}
		\begin{split}
			\int_\Omega |u_\eta(\cdot,t)|^{p-1}\operatorname{sign}(u_\eta(\cdot,t))\bv{b}_\eta\cdot\nabla u_\eta\,\dL{d} &= \int_{\partial\Omega}|h_\eta(\cdot,t)|^{p-1}\operatorname{sign}(h_\eta(\cdot,t))h_\eta(\cdot,t)(\bv{b}_\eta\cdot\nu_\Omega)\,\dH{d-1}\\
			&\quad - \int_\Omega u_\eta(\cdot,t)\operatorname{sign}(u_\eta(\cdot,t))|u_\eta(\cdot,t)|^{p-1}\divergence{\bv{b}_\eta}\,\dL{d}\\
			&\quad - \int_\Omega u_\eta(\cdot,t)\operatorname{sign}(u_\eta(\cdot,t))\nabla(|u_\eta(\cdot,t)|^{p-1})\cdot\bv{b}_\eta\,\dL{d}.
		\end{split}
	\end{align}
	For the third term on the right-hand side, we have
	\begin{multline}\label{eq:exst_6}
		\int_\Omega u_\eta(\cdot,t)\operatorname{sign}(u_\eta(\cdot,t))\nabla(|u_\eta(\cdot,t)|^{p-1})\cdot\bv{b}_\eta\,\dL{d}\\ = (p-1)\int_\Omega u_\eta(\cdot,t)\operatorname{sign}(u_\eta(\cdot,t))^2|u_\eta(\cdot,t)|^{p-2}\bv{b}_\eta\cdot\nabla u_\eta\,\dL{d}
		\\= (p-1)\int_\Omega |u_\eta(\cdot,t)|^{p-1}\operatorname{sign}(u_\eta(\cdot,t))\bv{b}_\eta\cdot\nabla u_\eta\,\dL{d}.
	\end{multline}
	Combining \eqref{eq:exst_5} and \eqref{eq:exst_6} gives
	\begin{multline*}
		p\int_\Omega |u_\eta(\cdot,t)|^{p-1}\operatorname{sign}(u_\eta(\cdot,t))\bv{b}_\eta\cdot\nabla u_\eta\,\dL{d} \\= \int_{\partial\Omega} |h_\eta(\cdot,t)|^p(\bv{b}_\eta\cdot\nu_\Omega)\,\dH{d-1} - \int_\Omega |u_\eta(\cdot, t)|^p\divergence{\bv{b}_\eta}\,\dL{d}.
	\end{multline*}
	We substitute the above equality into \eqref{eq:exst_4} to obtain
	\begin{equation*}
		\frac{d}{dt}\|u_\eta(\cdot, t)\|^p_p = \int_\Omega |u_\eta(\cdot, t)|^p\divergence{\bv{b}_\eta}\,\dL{d} - \int_{\partial\Omega} |h_\eta(\cdot, t)|^p(\bv{b}_\eta\cdot\nu_\Omega)\,\dH{d-1}.
	\end{equation*}
	Since $\divergence{\bv{b}}$ and $\bv{b}$ are essentially bounded, $\divergence{\bv{b}_\eta}$ and $\bv{b}_\eta$ are also essentially bounded uniformly in $\eta$.
	Moreover, $h_\eta$ converges to $h$ in $L^1(0,T;L^p(\partial\Omega))$, and hence $\|h_\eta(\cdot,t)\|_p^p$ is uniformly bounded in $\eta$.
	Hence, we deduce that
	\begin{equation}\label{eq:exst_7}
		\frac{d}{dt}\|u_\eta(\cdot,t)\|^p_p\leq M_1(\bv{b})\|u_\eta(\cdot,t)\|^p_p + M_2(\bv{b},h)
	\end{equation}
	for some positive constants $M_1(\bv{b})$ and $M_2(\bv{b},h)$ which are independent of $\eta$.
	Integrating \eqref{eq:exst_7} over $(0,t)$ gives
	\begin{equation*}
		\|u_\eta(\cdot, t)\|^p_p \leq \|u_\eta(\cdot,0)\|^p_p + M_1(\bv{b})\int_0^t\|u_\eta(\cdot, s)\|^p_p\,ds + M_2(\bv{b},h)t.
	\end{equation*}
	Since $\|u_\eta(\cdot,0)\|^p_p$ is also bounded by $u_\eta(\cdot,0)\to u_0$ in $L^p(\Omega)$, we deduce from the Gronwall inequality that
	\begin{equation*}
		\|u_\eta(\cdot, t)\|^p_p \leq M_2(\bv{b},h,u_0,T)e^{M_1(\bv{b})T}\qquad\mbox{for all}\quad t\in[0,T].
	\end{equation*}
	The above inequality implies that $\{u_\eta\}_{\eta > 0}$ is uniformly bounded in $L^p(0,T;L^p(\Omega))$, and hence we can extract
	a subsequence $\{u_{\eta_k}\}_{k\in\mathbb{N}}$ of $\{u_\eta\}_{\eta > 0}$ with $\eta_k\downarrow 0$ which converges to a function $\widetilde{u}$ weakly in $L^p(0,T;L^p(\Omega))$.
	Testing any $\varphi\in C^1_0(\closure{\Omega}\times[0,T))$, we observe that
	\begin{equation*}
		-\int_0^T\int_\Omega u_\eta\partial_t\varphi\,dxdt - \int_\Omega u_{0,\eta}\varphi(\cdot,0)\,dx + \int_0^T\int_{\partial\Omega} h_\eta(\bv{b}_\eta\cdot\nu_\Omega)\varphi\,dx'dt - \int_0^T\int_\Omega u_\eta\divergence{(\varphi\bv{b}_\eta)}\,dxdt = 0.
	\end{equation*}
	Sending $\eta\to 0$, up to a subsequence, we finally derive
	\begin{equation*}
		-\int_0^T\int_\Omega \widetilde{u}\partial_t\varphi\,dxdt - \int_\Omega u_0\varphi(\cdot,0)\,dx + \int_0^T\int_{\partial\Omega} h(\bv{b}\cdot\nu_\Omega)\varphi\,dx'dt - \int_0^T\int_\Omega \widetilde{u}\divergence{(\varphi\bv{b})}\,dxdt = 0.
	\end{equation*}
	The function $\widetilde{u}$ actually belongs to $L^\infty(0,T;L^p(\Omega))$ thanks to the lower semicontinuity of the weak convergence in $L^p(0,T;L^p(\Omega))$, say
	$\|\widetilde{u}(\cdot, t)\|_p\leq \liminf_{\eta\to 0}\|u_\eta\|_{L^p(\Omega\times(0,T))} < \infty$.
	Therefore, we conclude that $\widetilde{u}$ is the desired weak solution.
\end{proof}

{As a conclusion of this section, we shall show the existence of a weak solution as a corollary of} the relabeling lemma (\Lemma{LIn}) which
allows any weak solution to be renormalized by any smooth relabeling function with bounded derivative,
and we can invoke a compactness argument to obtain a weak convergent sequence of approximate solutions.

\begin{cor}\label{cor:exst_1}
	Assume that $\Omega$ is a smooth bounded domain in $\mathbb{R}^d$, and that
	$u_0\in L^\infty(\Omega)$, $h\in L^\infty(\Omega\times(0,T))$, {$\bv{b} \in L^1(0,T;W^{1,q}(\Omega;\mathbb{R}^d))$} and $\divergence{\bv{b}}\in L^\infty(\Omega\times(0,T))$.
	Then, there exists a weak solution $u\in L^\infty(0,T;L^p(\Omega))$ to \eqref{ETr}, \eqref{EDir} and \eqref{EInit}.
\end{cor}
\begin{proof}
	For each $\eta > 0$, let $\bv{b}_\eta$, $h_\eta$ and $u_{0,\eta}$ be smooth approximations of $\bv{b}$, $h$ and $u_0$ in terms of the mollifiers.
	Precisely, we set
	\begin{equation*}
		\bv{b}_\eta(x,t) := \int_0^\infty\conv{\bv{b}}{\mollDHalf{d}{\eta}}(x,s)\mollHalf{\eta}(s-t)\,ds,\quad h_\eta := \convS{h}{\mollDTHalf{d}{\eta}},\quad\mbox{and}\quad u_{0,\eta} := \convS{u_0}{\mollDHalf{d}{\eta}}.
	\end{equation*}
	Then, there exists a classical solution $u_\eta$ to the following problem:
	\begin{equation*}
		\begin{cases}
			u_\eta + \bv{b}_\eta\cdot\nabla u_\eta = 0&\qquad\mbox{in}\quad\Omega\times(0,T),\\
			u_\eta = h_\eta&\qquad\mbox{on}\quad\partial\Omega\times(0,T),\\
			u_\eta(\cdot,0) = u_{0,\eta}&\qquad\mbox{in}\quad\Omega.
		\end{cases}
	\end{equation*}
	We now take a relabeling function $\theta\in L^\infty(\mathbb{R})\cap C^1(\mathbb{R})$ with $\theta'\in L^\infty(\mathbb{R})$ which is
	strictly increasing and satisfies $\theta(0) = 0$. For instance, we may take $\theta(\sigma) := \tanh{\sigma}$ for $\sigma\in\mathbb{R}$.
	Then, it is easily observed that $\theta(u_\eta)$ is the classical solution to the following problem:
	\begin{equation*}
		\begin{cases}
			(\theta(u_\eta))_t + \bv{b}_\eta\cdot\nabla\theta(u_\eta) = 0&\qquad\mbox{in}\quad\Omega\times(0,T),\\
			\theta(u_\eta) = \theta(h_\eta)&\qquad\mbox{on}\quad\partial\Omega\times(0,T),\\
			\theta(u_\eta)(\cdot,0) = \theta(u_{0,\eta})&\qquad\mbox{in}\quad\Omega.
		\end{cases}
	\end{equation*}
	Since $\theta(u_\eta)$ is a weak solution to the above problem, for every $\varphi\in C^1_0(\closure{\Omega}\times[0,T))$, {we have}
	\begin{multline}\label{eq:exst_1}
		-\int_0^T\int_\Omega \theta(u_\eta)\partial_t\varphi\,dxdt - \int_\Omega\theta(u_{0,\eta})\varphi(\cdot,0)\,dx\\ + \int_0^T\int_{\partial\Omega}\theta(h_\eta)(\bv{b}_\eta\cdot\nu_\Omega)\varphi\,d\mathcal{H}^{d-1}dt - \int_0^T\int_\Omega \theta(u_\eta)\divergence{(\varphi\bv{b}_\eta)}\,dxdt = 0.
	\end{multline}
	Since $\Omega$ and $\theta$ are bounded, $\theta(u_\eta)$ is also bounded in $L^p(0,T;L^p(\Omega))$ uniformly in $\eta$,
	and hence we can extract a subsequence $\{\theta(u_{\eta_k})\}_{k\in\mathbb{N}}$ which converges to a function $\widetilde{u}$ weakly in $L^p(0,T;L^p(\Omega))$.
	Up to a subsequence, letting $\eta\to 0$ gives
	\begin{multline}\label{eq:exst_2}
		-\int_0^T\int_\Omega\widetilde{u}\partial_t\varphi\,dxdt - \int_{\Omega}\theta(u_0)\varphi(\cdot,0)\,dx\\ + \int_0^T\int_{\partial\Omega}\theta(h)(\bv{b}\cdot\nu_\Omega)\varphi\,d\mathcal{H}^{d-1}dt - \int_0^T\int_\Omega\widetilde{u}\divergence{(\varphi\bv{b})}\,dxdt = 0.
	\end{multline}
	Here, the convergence of the second term is deduced from the Lebesgue dominated convergence theorem together with
	\begin{equation*}
		\theta(u_{0,\eta})\varphi(\cdot,0) \leq \|\theta'\|_{L^\infty(\mathbb{R})}\sup_{\closure{\Omega}}|\varphi|
	\end{equation*}
	and $\theta(u_{0,\eta})\to \theta(u_0)$ a.e. in $\Omega$ as $\eta\to 0$ since $u_{0,\eta} \to u_0$ a.e. in $\Omega$ and $\theta$ is continuous.
	For the third term, we compute
	\begin{align*}
		&\left|\int_0^T\int_{\partial\Omega}\theta(h_\eta)(\bv{b}_\eta\cdot\nu_\Omega)\varphi\,\dH{d-1}dt - \int_0^T\int_{\partial\Omega}\theta(h)(\bv{b}\cdot\nu_\Omega)\varphi\,\dH{d-1}dt\right|\\
		&\leq \left|\int_0^T\int_{\partial\Omega}\theta(h_\eta)(\bv{b}_\eta\cdot\nu_\Omega)\varphi\,\dH{d-1}dt - \int_0^T\int_{\partial\Omega}\theta(h)(\bv{b}_\eta\cdot\nu_\Omega)\varphi\,\dH{d-1}dt\right|\\
		&\quad + \left|\int_0^T\int_{\partial\Omega}\theta(h)(\bv{b}_\eta\cdot\nu_\Omega)\varphi\,\dH{d-1}dt - \int_0^T\int_{\partial\Omega}\theta(h)(\bv{b}\cdot\nu_\Omega)\varphi\,\dH{d-1}dt\right|\\
		&= \left|\int_0^T\int_{\partial\Omega}{\left(\int_0^1\theta'(sh_\eta - (1-s)h)\,ds(h_\eta - h)\right)}(\bv{b}_\eta\cdot\nu_\Omega)\varphi\,\dH{d-1}dt\right| \\
		&\quad + \left|\int_0^T\int_{\partial\Omega}\theta(h)(\bv{b}_\eta\cdot\nu_\Omega - \bv{b}\cdot\nu_\Omega)\varphi\,\dH{d-1}dt\right|\\
		&\leq \|\theta'\varphi\|_\infty\|h_\eta - h\|_p\|\bv{b}_\eta\cdot\nu_\Omega\|_{p'} + \|\theta\varphi\|_\infty\|\bv{b}_\eta\cdot\nu_\Omega - \bv{b}\cdot\nu_\Omega\|_{p'}\mH{d-1}(\partial\Omega)^{\frac{1}{p}}\to 0\quad\mbox{as}\quad\eta\to 0.
	\end{align*}
	Here, we have invoked the H\"{o}lder inequality to derive the last inequality.
	Since $\theta(u_\eta)\rightharpoonup\widetilde{u}$ weakly in $L^\infty(0,T;L^p(\Omega))$,
	$\bv{b}_\eta\to \bv{b}$ and $\divergence{\bv{b}_\eta}\to \divergence{\bv{b}}$ strongly in $L^{p'}(\Omega\times(0,T))$ as $\eta\to 0$,
	the product of them weakly converges to the corresponding product, and hence the convergence of the fourth term follows.

	The integral equation \eqref{eq:exst_2} implies that $\widetilde{u}$ is a weak solution to \eqref{ETr}, \eqref{EDir} and \eqref{EInit}
	with the Dirichlet data $\theta(h)$ and the initial data $\theta(u_0)$.

	Letting $C := \max\{\|u_0\|_\infty, \|h\|_\infty\}$, we now define another relabeling function $\widetilde{\theta}$ by
	\begin{equation*}
		\widetilde{\theta}(\sigma) := \begin{cases}
			\theta^{-1}(\sigma)\qquad&\mbox{if}\quad\sigma\in[\theta(-C),\theta(C)],\\
			\frac{1}{\theta'(C)}(\sigma - \theta(C)) + C\qquad&\mbox{if}\quad\sigma\in(\theta(C),\infty),\\
			\frac{1}{\theta'(-C)}(\sigma - \theta(-C)) -C\qquad&\mbox{if}\quad\sigma\in(-\infty,\theta(-C)).
		\end{cases}
	\end{equation*}
	Roughly speaking, the function $\widetilde{\theta}$ corresponds to the inverse function of $\theta$ in the interval $(-\theta(C),\theta(C))$.
	Meanwhile, it is linearly extended to the outside of the interval (see \Figure{fig:theta}).

    \begin{figure}[H]
        \centering
		\includegraphics[keepaspectratio,width=80mm]{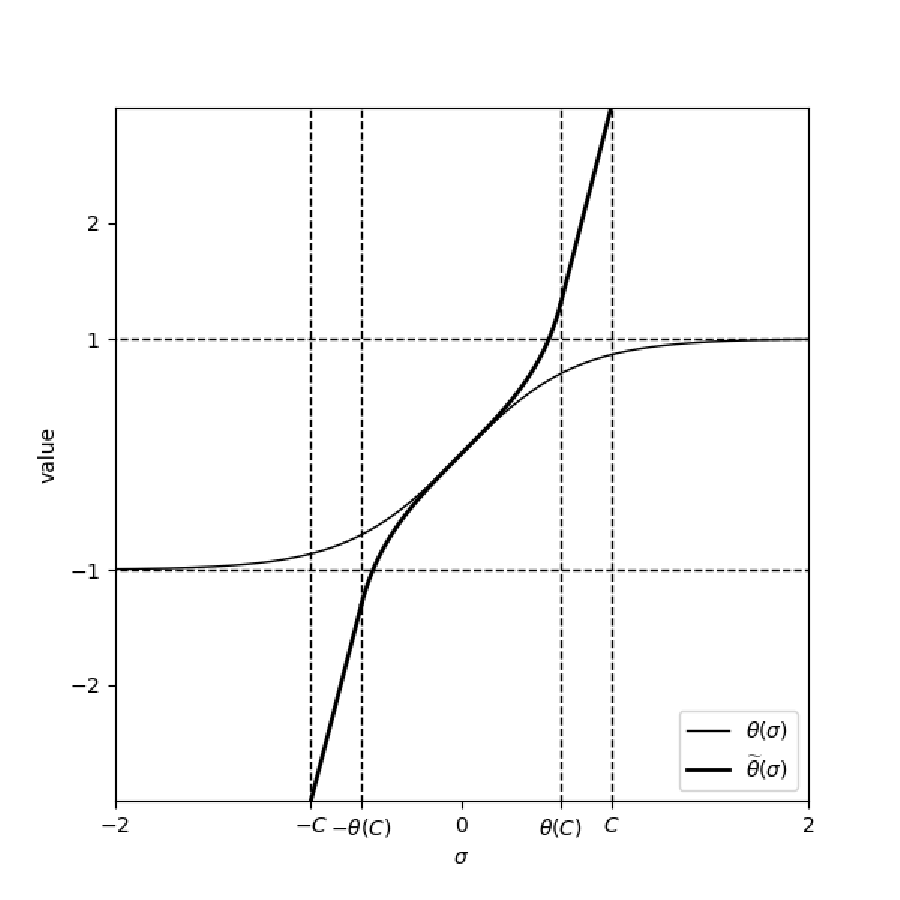}
		\caption{The graphs of $\theta$ and $\widetilde{\theta}$ for $\theta(\sigma) = \tanh{\sigma}$.}
		\label{fig:theta}
    \end{figure}

	Let us show that $\widetilde{\theta}(\widetilde{u})$ is a weak solution to \eqref{ETr}, \eqref{EDir} and \eqref{EInit}.
	Indeed, we easily observe that $\widetilde{\theta}\in C^1(\mathbb{R})$ and $\widetilde{\theta}'\in L^\infty(\mathbb{R})$,
	and hence $\widetilde{\theta}$ is available as a relabeling function for \Lemma{LIn}.
	Moreover, we note that $\widetilde{\theta}(\widetilde{u})\in L^\infty(0,T;L^p(\Omega))$ since
	\begin{equation*}
		\|\widetilde{\theta}(\widetilde{u})\|_p\leq \|\widetilde{\theta}'\|_\infty\|\widetilde{u}\|_p,\quad\mbox{and}\quad \|\widetilde{u}\|_p\leq \liminf_{\eta\to 0}\|\theta(u_\eta)\|_p < \infty.
	\end{equation*}
	Here, we have used the lower semicontinuity of the norm in $L^p(\Omega)$ with respect to the weak convergence $\theta(u_\eta)\rightharpoonup\widetilde{u}$.
	Therefore, we may adopt \Lemma{LIn} with $\theta = \widetilde{\theta}$ together with \eqref{eq:exst_2} to conclude that
	$\widetilde{\theta}(\widetilde{u})$ is the desired weak solution.
\end{proof}
\section{Conclusion}
In this study, we have established a notion of a weak solution to the time-dependent transport equation
with inhomogeneous Dirichlet boundary conditions. For each Dirichlet boundary data (irrelevant on the inflow place) and the initial data,
this notion ensures the uniqueness of a weak solution to the problem
provided that the divergence of the transport vector field is essentially bounded.
The uniqueness of the weak solution is guaranteed by the relabeling lemma,
and this lemma has been shown by approximation of the weak solution in terms of the mollification tailored to the {inhomogeneous} boundary problems.
Using the interchanging properties of the commutator, we have obtained the explicit forms of the Dirichlet data and the initial data of the approximate solutions.
For bounded smooth domains, bounded Dirichlet data, and bounded initial data,
the existence result is deduced using a renormalizing process which was originally developed to show the uniqueness of the weak solution.

\section{Acknowledgments}
The work of the second author was partly supported by the Japan Society 
for the Promotion of Science (JSPS) through the grant Kakenhi: 
No.~24K00531, and by Arithmer Inc., Daikin Industries, Ltd.\ 
and Ebara Corporation through collaborative grants.
\bibliographystyle{siam}
\bibliography{cite.bib}
\end{document}